\documentclass[12pt]{amsart}
\usepackage{amsmath,amssymb,amscd,amsthm,amsxtra}
\usepackage{latexsym}

\usepackage[dvips]{graphicx}        

\newtheorem{theorem}{Theorem}[section]
\newtheorem{lemma}{Lemma}[section]
\newtheorem{proposition}{Proposition}[section]
\newtheorem{remark}{Remark}[section]

\newtheorem{definition}{Definition}[section]
\newtheorem{corollary}{Corollary}[section]

\newcommand{\La}{\Lambda}

\newcommand{\ti}{\tilde}

\newcommand{\wh}{\widehat}

\newcommand{\ZR}{\mathbb{R}}
\newcommand{\ZZ}{\mathbb{Z}}
\newcommand{\e}{\varepsilon}
\newcommand{\w}{\omega}
\newcommand{\Id}{{\bf 1}}

\newcommand{\bS}{{\bf S}}
\newcommand{\bT}{{\bf T}}

\newcommand{\bQ}{{\bf Q}}

\newcommand{\bI}{{\bf I}}

\newcommand{\mQ}{{\mathcal Q}}

\begin{document}

\title
[Bilinear Hilbert Transforms along curves, I]
{Bilinear Hilbert Transforms along curves \\I. The monomial case}

\author{Xiaochun Li}

\address{
Xiaochun Li\\
Department of Mathematics\\
University of Illinois at Urbana-Champaign\\
Urbana, IL, 61801, USA}

\email{xcli@math.uiuc.edu}

\date{\today}

\thanks{Research was partially supported by the NSF}


\begin{abstract}
We establish an $L^2\times L^2$ to $L^1$ estimate for the bilinear
Hilbert transform along a curve defined by a monomial. Our proof is closely
related to multilinear oscillatory integrals. 
\end{abstract}
\maketitle

\section{Introduction}
\setcounter{equation}0

Let $d\geq 2$ be a positive integer.
We consider the bilinear Hilbert transform along a curve $\Gamma(t)=(t, t^d)$ defined by
\begin{equation}\label{defofH}
H_{\Gamma}(f,g)(x)={p.v.}\int_{\ZR}f(x-t)g(x - t^d)\frac{dt}{t}\,,
\end{equation}
where $f, g$ are Schwartz functions on $\ZR$.

The main theorem we prove in this paper is

\begin{theorem}\label{thm1}
The bilinear Hilbert transform along the curve $\Gamma(t)=(t, t^d)$ can be extended to a bounded
operator from $L^2\times L^2$ to $L^1$.
\end{theorem}

\begin{remark}
It can be shown, with a little modification of our method, that the bilinear Hilbert transforms
along polynomial curves $(t, P(t))$ are bounded from $L^p\times L^q$ to $L^r$ whenever $(1/p, 1/q, 1/r)$ is
in the closed convex hull of $(1/2, 1/2, 1) $, $(1/2, 0, 1/2) $ and $(0, 1/2, 1/2)$.
\end{remark}

This problem is motivated by the Hilbert transform along a curve $\Gamma=(t, \gamma(t))$
 defined by
$$
H_{\Gamma}(f)(x_1,x_2)={p.v.}\int_{\ZR}f(x_1-t, x_2-\gamma(t))\frac{dt}{t}\,,
$$
and the bilinear Hilbert transform defined by
$$
H(f, g)(x) ={p.v.}\int_{\ZR}f(x-t)g(x+t)\frac{dt}{t}\,.
$$
Among various curves, one simple model case is the parabola $(t,t^{2})$ in the two
dimensional plane. This work was initiated by Fabes and Riviere \cite{FR} in
order to study the regularity of parabolic differential equations.
In the last thirty years, considerable work on this type of problems had been done. 
A nice survey on this type of operators was written by  Stein and Wainger \cite{SW}.
For the curves on homogeneous nilpotent Lie groups, the $L^p$ estimates
were established by Christ \cite{Christ1}.  The work for the Hilbert transform along
more general curves with certain geometric conditions such as the "flat" case
can be found in Christ, Duoandikoetxea and J. L. Rubio de Francia, and Nagel,
Vance, Wainger and Weinberg's papers \cite{DRubio,Christ2, NVWW}.
The general results were established recently in \cite{CNSW} for the singular Radon transforms and 
their maximal analogues over smooth submanifolds of ${\mathbb{R}}^n$ with some curvature 
conditions.\\

In recent years there has been a very active trend of harmonic analysis
using time-frequency analysis to deal with multi-linear operators. 
A breakthrough on the bilinear Hilbert transform was made 
by Lacey and Thiele \cite{LT1, LT2}. Following Lacey and Thiele's work, 
the field of multi-linear 
operators has been actively developed, to the point that some of the most 
interesting open questions have a strong connection to some kind of non-abelian 
analysis. For instance, the tri-linear Hilbert transform 
\begin{equation*}
p.v.\int f_1 (x+t) f_2 (x+2t) f _3 (x+3t) \frac{dt}{t} 
\end{equation*}
has a hidden quadratic modulation symmetry which must be accounted for in any 
proposed method of analysis. This non-abelian character is explicit in the 
work of B.~Kra and B.~Host \cite{HK} who characterize the 
characteristic factor of the corresponding ergodic averages 
\begin{equation*}
N ^{-1} \sum _{n=1} ^{N} f_1 (T ^{n}) f_2 (T ^{2n}) f_3 (T ^{3n}) 
\longrightarrow \prod _{j=1} ^{3} \mathbb E (f_j \mid \mathcal N)
\end{equation*}
Here, $ (X, \mathcal A, \mu ,T )$ is a measure preserving system, $ \mathcal N\subset
\mathcal A$ is the sigma-field which describes the characteristic factor.  In this case, 
it arises from certain $ 2$-step nilpotent groups.  The limit above is in the sense 
of $ L ^{2}$-norm convergence, and holds for all bounded $ f_1,f_2,f_3$.

The ergodic analog of  the bilinear Hilbert transform along a parabola is 
the non-conventional bi-linear average  
\begin{equation*}
N ^{-1} \sum _{n=1} ^{N} f_1 (T ^{n}) f_2 (T ^{n ^2 })
\longrightarrow \prod _{j=1} ^{2} \mathbb E (f_j \mid \mathcal K _{\textup{profinite}})
\end{equation*}
where $ \mathcal K _{\textup{profinite}}\subset \mathcal A$ is the profinite factor, 
a subgroup of the maximal abelian factor of $ (X, \mathcal A, \mu ,T)$.  
The proof  of the characteristic factor result  
above, due to Furstenberg \cite{Fur}, utilizes the characteristic factor for the 
three-term result.  We are indebted to M. Lacey for bringing Furstenberg's theorems to our attention.
However, a notable fact is that our proof for the bilinear Hilbert transform along a monimial
curve does not have to go through the tri-linear Hilbert transform. The proof provided in 
this article  heavily relies on the concept of "quadratic" uniformity and some kind 
of "quadratic" Fourier analysis, initiated by Gowers \cite{G}. And perhaps this is a starting point 
to understand the tri-linear Hilbert transform.  \\


Another prominent theme is the relation of the bilinear Hilbert transforms along curves and 
the multilinear oscillatory integrals. The bilinear Hilbert transforms along curves are closely 
associated to the multilinear oscillatory integrals of the following type.
\begin{equation}
 \Lambda_\lambda(f_1, f_2, f_3)=\int_{\bf B} f_1({\bf x}\cdot {\bf v}_1) f_2({\bf x}\cdot {\bf v}_2) f_3({\bf x}\cdot {\bf v}_3)
  e^{i\lambda \varphi({\bf x})} d{\bf x}\,,
\end{equation}
where ${\bf B}$ is a unit ball in $\mathbb R^3$, ${\bf v}_1, {\bf v}_2, {\bf v}_3$ are vectors in $\mathbb R^3$,
and the phase function $\varphi$ satisfies a non-degenerate condition 
\begin{equation}\label{nondegcond}
 \left|\prod_{j=1}^3\left(\nabla\cdot {\bf v}_j^\perp \right)\varphi({\bf x}) \right|\geq 1\,.
\end{equation}
Here ${\bf v}_j^\perp$'s are unit vectors orthogonal to ${\bf v}_j$'s respectively. For a polynomial 
phase $\varphi$ with the non-degenerate condition (\ref{nondegcond}), it was proved in \cite{CLTT}
that
\begin{equation}\label{Linftyest}
 \left| \Lambda_\lambda(f_1, f_2, f_3) \right|\leq C(1+|\lambda|)^{-\e}\prod_{j=1}^3\|f_j\|_\infty\,
\end{equation} 
holds for some positive number $\e$. For the particular vectors ${\bf v}_j$'s and the non-degenerate phase
$\varphi$ encountered in our problem, an estimate similar to (\ref{Linftyest}) still holds. However, one of the main
 difficulties arises from the falsity of $L^2$ decay estimates for the trilinear form $\Lambda_\lambda$.
In order to overcome this difficulty, we end up introducing the "quadratic" uniformity, which
plays a role of a "bridge" connecting two spaces $L^2$ and $L^\infty$.   \\

The method used in this paper essentially works for those curves on nilpotent
groups. It is possible to extend Theorem {\ref{thm1}} to the general setting of
nilpotent Lie groups. But we will not pursue this in this article.
There are some related questions one can pose. Besides the generalisation to 
the more general curves, it is natural to ask the corresponding problems 
in higher dimensional cases and/or in multi-linear cases. 
For instance, in the tri-linear case, one can consider 
\begin{equation}
 T(f_1, f_2, f_3)(x)= {p.v.}\int f_1(x+t)f_2(x+p_1(t))f_3(x+p_2(t)) \frac{dt}{t}\,. 
\end{equation}
Here $p_1, p_2$ are polynomials of $t$. The investigation of such problems
will be discussed in subsequent papers.\\

\noindent
{\bf Acknowledgement} The author would like to thank 
his wife, Helen,  and his son, Justin, for being together through
the hard times in the past two years. And he is also very grateful to 
Michael Lacey for his constant support and encouragement.

\section{A Lemma and A Counterexample}\label{sectionfixj}
\setcounter{equation}0

Let $\rho$ be a Schwartz function supported in the union of two intervals
$[-2, -1/2]$ and $[1/2, 2]$.

\begin{lemma}\label{lemmaj}
Let $P$ be a real polynomial with degree $d\geq 2$. And let $2 \leq
n \leq d$. Suppose that the $n$-th order derivative of $P$,
$P^{(n)}$, does not vanish. Let $T(f,g)(x)= \int f(x-t)g(x-P(t))
\rho(t) dt$. Then $T$ is bounded from $L^{p}\times L^q $ to $L^r$
for $p, q>1$, $r> \frac{n-1}{n}$ and $ 1/p+1/q=1/r$.
\end{lemma}
\begin{proof}
We may without loss of generality restrict $x$, hence likewise the
supports of $f, g$, to fixed bounded intervals whose sizes depend on
the coefficients of the polynomial $P$. This is possible because of
the restriction $|t|\leq 2$ in the integral. Let us restrict $x$ in
a bounded interval $I_P$. It is obvious that $T$ is bounded
uniformly from $L^\infty\times L^\infty$ to $L^\infty$ and from
$L^p\times L^{p'}$ to $L^1$ for $1\leq p\leq\infty$ and
$1/p+1/p'=1$. When $P'(t)\neq 1$ in $1/2\leq |t|\leq 2$, then the
boundedness from $L^1\times L^1$ to $L^1$ can be obtained
immediately by changing variable $u=x-t$ and $v=x-P(t)$ since the
Jocobian $\frac{\partial (u,v)}{\partial (x,t)}= 1-P'(t)$. Thus $T$
is bounded from $L^1\times L^1$ to $L^{1/2}$ since $x$ is restricted
to a bounded interval $I_P$ and then the lemma follows by
interpolation. When there is a real solution in $1/2\leq|t| \leq 2$
to the equation $P'(t)=1$, the trouble happens at a neighborhood of
$t_0$, where $t_0\in \{t: 1/2\leq |t|\leq 2\}$ is the real solution to
$P'(t)=1$. There are at most $d-1$ real solutions to the equation
$P'(t)-1=0$. Thus we only need to consider a small neighborhood
containing only one real solution $t_0$ to $P'(t)=1$. Let $I(t_0)$
be a small neighborhood of $t_0$ which contains only one real
solution to $P'(t)-1=0$. We should prove that
\begin{equation}\label{estlem0}
\int_{I_P}\bigg|\int_{I(t_0)} f(x-t)g(x-P(t))\rho(t) dt\bigg|^r dx
\leq C_P\|f\|_{p}^r\|g\|_q^r\,,
\end{equation}
for $p>1,q>1$ and $r>(n-1)/n$ with $1/p+1/q=1/r$. Let $\rho_0$ be a
suitable bump function supported in $1/2\leq|t|\leq 2$ such that
$\sum_j\rho_0(2^jt)=1$. To get (\ref{estlem0}), it suffices to prove
that there is a positive $\e$ such that
\begin{equation}\label{estlem1}
\int_{I_P}\bigg|\int_{I(t_0)} f(x-t)g(x-P(t))\rho(t)\rho_0({2^j(t-t_0)})dt \bigg|^r dx
\leq C2^{-\e j}\|f\|_{p}^r\|g\|_q^r\,,
\end{equation}
for all large $j$, $p>1,q>1$ and $r>(n-1)/n$ with $ 1/p+1/q=1/r$,
since (\ref{estlem0}) follows by summing for all possible $j\geq 1$.
By a translation argument we need to show that
\begin{equation}\label{estlem11}
\int_{I_P}\bigg|\int  f(x-t)g(x-P_1(t))\rho_0({2^jt})dt \bigg|^r dx
\leq C2^{-\e j}\|f\|_{p}^r\|g\|_q^r\,,
\end{equation}
for all large $j$, $p>1,q>1$ and $r>(n-1)/n$ with $ 1/p+1/q=1/r$,
where $P_1$ is a polynomial of degree $d$ defined by
$P_1(t)=P(t+t_0)-P(t_0)$. It is clear that $P'_1(0)=1$ and
$P_1^{(n)}\neq 0$.  When $|t|\leq 2^{-j+1}$, $|P_1(t)|\leq C_P
2^{-j}$ for some constant $C_P\geq 1$ depending on the coefficients
of $P$. Let $I_P=[a_P, b_P]$ and $A_N$ be defined by
$$
A_N= [a_P + NC_P2^{-j}, a_P+(N+1)C_P2^{-j}] \,\,{\rm for}\,\, N=-1,
\cdots, \frac{(b_P-a_P)\cdot 2^{j}}{C_P}\,.
$$
Notice that for a fixed $x\in I_P$, $x-t, x-P_1(t)$ is in
$A_{N-1}\cup A_N \cup A_{N+1}$ for some $N$. So we can restrict $x$ in
one of $A_N$'s. Now let $T_N(f,g)(x)=1_{A_N}(x)\int
f(x-t)g(x-P_1(t))\rho_0(2^jt)dt $. Due to the restriction of $x$, 
 we only need to show that
\begin{equation}\label{estlem3}
 \|T_N(f,g)\|_r^r
\leq C2^{-\e j}\|f\|_{p}^r\|g\|_q^r\,
\end{equation}
for all large $j\geq 1$, $p>1,q>1$ and $r>(n-1)/n$ with $
1/p+1/q=1/r$, where $f_N = f1_{A_N}$, $g_N=g1_{A_N}$ and $C$ is
independent of $N$.

By inserting absolute values throughout
we get $T_N$ maps $L^p\times L^q$ to $L^r$ with a bound $C2^{-j}$
uniform in $N$, whenever $(1/p, 1/q, 1/r)$ belongs to the closed
convex hull of the points $(1,0,1)$, $(0,1,1)$ and $(0,0,0)$.
Observe that
$P_1'(t)=1+\sum_{k=2}^{d-1}\frac{P^{(k)}_1(0)}{(k-1)!}t^{k-1}$ since
$P'_1(0)=1$. By $P_1^{(n)}(0)\neq 0$ and applying Cauchy-Schwarz
inequality, we obtain for all $j$ large enough,
\begin{eqnarray*}
 & &  \int  \big| T_N(f, g)(x)\big|^{1/2} dx\\
 & \leq & C_P 2^{-j/2} \|T_{N}(f,g)\|_1^{1/2}\\
 & \leq & C_P 2^{-j/2} 2^{(n-1)j/2}\|f\|_1^{1/2}\|g\|_1^{1/2} \,=
 C_P 2^{(n-2)j/2}\|f\|_1^{1/2}\|g\|_1^{1/2}
\,.
\end{eqnarray*}
Hence an interpolation then yields a bound $C2^{-\e j}$ for all
triples of reciprocal exponents within the convex hull of $(1,
\frac{1}{n-1}, \frac{n}{n-1})$, $(\frac{1}{n-1}, 1, \frac{n}{n-1})$,
$(1,0,1)$, $(0,1,1)$ and $(0,0,0)$. This finishes the proof of
(\ref{estlem3}). Therefore we complete the proof of Lemma
\ref{lemmaj}
\end{proof}

Notice that if $P$ is a monomial $t^d$, then the lower bound for $r$
in Lemma \ref{lemmaj} can be improved to $1/2$. This is because $P_1(t)=
P(t+t_0)-P(t_0)=(t+t_0)^d-t_0^d$ has nonvanishing $P_1^{(2)}(0)$ when $1/2
\leq |t_0|\leq 1$. We now  give a counterexample to indicate that the lower bound $(n-1)/n$
for $r$ is sharp in Lemma \ref{lemmaj}. 
\begin{proposition}\label{counterexample}
Let $d, n$ be integers such that $d\geq 2$ and $2\leq n\leq d$.
There is a real polynomial $Q$ of degree $d \geq 2$ whose $n$-th
order derivative does not vanish such that $T_Q$ is unbounded from
$L^p\times L^q$ to $L^r$ for all $p, q>1$ and $r< (n-1)/n$ with
$1/p+1/q=1/r$, where $T_Q$ is the bilinear operator defined by
$T_Q(f,g)(x)=\int f(x-t)g(x-Q(t))\rho(t)dt$.
\end{proposition}
\begin{proof}
Let $A$ be a very large number.  We define $Q(t)$ by
\begin{equation}\label{defofQ}
 Q(t) = \frac{1}{Ad!}(t-1)^d + \frac{1}{An!}(t-1)^n  + (t-1)\,.
\end{equation}
It is sufficient to prove that if $T_Q$ is bounded from $L^p\times
L^q$ to $L^r$ for some $p, q>1$ and $1/r=1/p+1/q$, then $r \geq (n-1)/n$.
Suppose there is a constant $C$ such that $\|T_Q(f,g)\|_r\leq
C\|f\|_p\|g\|_q$ for all $f\in L^p$ and $g\in L^q$. Let $\delta$ be
a small positive number. And let $f_\delta=1_{[0,2^n\delta]}$ and
$g_\delta=1_{[1-\delta, 1]}$. Let $D_1$ be the intersection point of
the curves $x=Q(t)+1$ and $x=t+2^n\delta$ in $tx$-plane with $t>1$,
and let $D_2$ be the intersection point of the curves
$x=Q(t)+1-\delta$ and $x=t$ in $tx$-plane with $t>1$. Let $D_1=(t_1,
x_1)$ and $D_2 =(t_2, x_2)$. Then
$$1+2^{1-1/n}(An!)^{1/n}\delta^{1/n} \leq t_1 \leq
1+2(An!)^{1/n}\delta^{1/n}\,\, {\rm and}$$
$$1+2^{-1/n}(An!)^{1/n}\delta^{1/n} \leq t_2 \leq
1+ (An!)^{1/n}\delta^{1/n}\,.$$ Thus we have
$$
1+2^{1-1/n}(An!)^{1/n}\delta^{1/n} + 2^n\delta\leq x_1 \leq
1+2(An!)^{1/n}\delta^{1/n}+2^n\delta\,\, {\rm and}
$$
$$1+2^{-1/n}(An!)^{1/n}\delta^{1/n} \leq x_2 \leq
1+ (An!)^{1/n}\delta^{1/n}\,.$$ When $A$ is large and $\delta$ is
small, any horizontal line between line $x=x_1$ and line $x=x_2$ has
a line segment of length $\delta/2$  staying within the region
bounded by curves $x=t$, $x=Q(x)+1-\delta$, $x=t+2^n\delta$ and
$x=Q(t)+1$. Hence, we have
\begin{equation}\label{lhsbig}
\|T_Q(f_\delta, g_\delta)\|_r^r \geq (\delta/2)^r
(An!)^{1/n}\delta^{1/n}/100\,.
\end{equation}
By the boundedness of $T_Q$, we have
$$
\|T_Q(f_\delta, g_\delta)\|_r^r\, \leq C^r
(2^n\delta)^{r/p}\delta^{r/q}\,= C^r2^{nr/p}\delta\,.
$$
By (\ref{lhsbig}) we have
\begin{equation}\label{sharpest}
\delta^{r}\leq
\frac{1002^{r+nr/p}C^r}{(An!)^{1/n}}\delta^{\frac{n-1}{n}}\,.
\end{equation}
Since $A$ can be chosen to be a very large number and $\delta$ can
be very small, (\ref{sharpest}) implies $r\geq\frac{n-1}{n}$, which
completes the proof of Lemma {\ref{counterexample}}.
\end{proof}

\vspace{0.4cm}

\section{A Decomposition}\label{decomp}
\setcounter{equation}0

Let $\rho_1$ be a standard bump function supported on $[1/2, 2]$. And let
$$\rho(t) = \rho_1(t)\Id_{\{t>0\}} - \rho_1(-t)\Id_{\{t<0\}}\,.$$
It is clear that $\rho$ is an odd function.  To obtain the $L^r$ estimates
for $H_\Gamma$, it is sufficient to get $L^r$ estimates for
$T_\Gamma$ defined by $T_\Gamma =\sum_{j\in {\mathbb Z}} T_{\Gamma,
j}$, where $T_{\Gamma, j}$ is
\begin{equation}\label{defofTGj}
T_{\Gamma,j}(f, g)(x) = \int f(x-t)g(x- t^d)2^j\rho(2^jt)dt\,.
\end{equation}
Let $L$ be a large positive number (larger than $2^{100}$).
 By Lemma \ref{lemmaj}, we have that if $|j| \leq  L$,
$\|T_{\Gamma, j}(f,g)\|_{r} \leq C_L \|f\|_p\|g\|_q$ for all $p, q>1$ and $1/p+1/q=1/r$, where
the operator norm $C_L$ depends on the upper bound $L$. 
Hence in the following we only need to consider
the case when $|j|> L$.  In fact we prove the following theorem.

\begin{theorem}\label{thmlargej}
Let $T_{\Gamma, j}$ be
defined as in (\ref{defofTGj}). Then
the bilinear operator $T_L =\sum_{j\in\ZZ:|j|>L}T_{\Gamma,j}$ is
bounded from $L^2\times L^2$ to $L^1$. 
\end{theorem}

Clearly Theorem \ref{thm1} follows by Theorem \ref{thmlargej} and Lemma
\ref{lemmaj}. The rest part of the article is devoted to a proof
of Theorem \ref{thmlargej}.\\

We begin the proof of Theorem \ref{thmlargej} by constructing 
an appropriate decomposition of the operator $T_{\Gamma,j} $.  This is done 
by an analysis of the bilinear symbol associated with the operator. \\

Expressing $T_{\Gamma, j}$ in dual frequency variables, we have  
$$
T_{\Gamma, j}(f,g)(x)=\int\int {\wh f}(\xi){\wh g}(\eta)e^{2\pi
i(\xi+\eta)x}\mathfrak{m}_j(\xi, \eta)d\xi d\eta\,,
$$
where the symbol $\mathfrak{m}_j$ is defined by
\begin{equation}\label{defofmj}
\mathfrak{m}_j(\xi, \eta)=\int \rho(t) e^{-2\pi i \left(2^{-j}\xi t + 
 2^{-dj}\eta t^d\right)}\,dt\,.
\end{equation}

First we introduce a resolution of the identity. 
Let $\Theta $ be a Schwarz function supported on $(-1,1)$ such 
that $\Theta(\xi)=1$ if $|\xi|\leq 1/2$. 
Set $\Phi$ to be a Schwartz function satisfying 
$$\wh\Phi(\xi)=\Theta(\xi/2)-\Theta(\xi)\,.$$
Then $\Phi$ is a Schwartz function such that $\wh\Phi$ is supported
on $\{\xi: 1/2 < |\xi| < 2\}$ and 
\begin{equation}\label{defofphi}
\sum_{m\in\mathbb Z}\wh\Phi\big(\frac{\xi}{2^m}\big)=1\,\, {\rm for}
\,\,{\rm all}\,\, \xi\in \mathbb R\backslash \{0\}\, ,
\end{equation}
and for any $m_0\in \mathbb Z$, 
\begin{equation}\label{defofphi0}
{\wh \Phi_{m_0}}(\xi) := \sum_{m=-\infty}^{m_0}
\wh\Phi{\big(\frac{\xi}{2^m}\big)} = \Theta\big(\frac{\xi}{2^{m_0+1}}\big)\,,
\end{equation}
which is a bump function supported on $(-2^{m_0+1}, 2^{m_0+1})$.

From (\ref{defofphi}), we can decompose $T_{\Gamma, j}$ into two parts: 
$T_{\Gamma,j,1}  $  and $T_{\Gamma,j,2}$, where 
$T_{\Gamma,j,1}$ is given by
\begin{equation}\label{defofTj1}
\sum_{m\in \mathbb Z}\sum_{\substack{m'\in \ZZ:\\|m'-m|> 10^d}}
\iint {\wh f}(\xi){\wh g}(\eta)e^{2\pi
i(\xi+\eta)x}\wh\Phi\big(\frac{2^{-j}\xi}{2^m}\big)
\wh\Phi\big(\frac{2^{-d j}\eta}{2^{m'}}\big)\mathfrak{m}_j(\xi,
\eta)d\xi d\eta\,,
\end{equation}
and 
$T_{\Gamma,j,2}$ is defined by
\begin{equation}\label{defofTj2}
\sum_{m\in \mathbb Z}\sum_{\substack{m'\in \ZZ:\\|m'-m|\leq 10^d}}
\iint {\wh f}(\xi){\wh g}(\eta)e^{2\pi
i(\xi+\eta)x}\wh\Phi\big(\frac{2^{-j}\xi}{2^m}\big)
\wh\Phi\big(\frac{2^{-d j}\eta}{2^{m'}}\big)\mathfrak{m}_j(\xi,
\eta)d\xi d\eta\,.
\end{equation}

Define $\mathfrak{m}_d$ by 
\begin{equation}\label{defofmfmd}
\mathfrak{m}_d(\xi, \eta) = \int \rho(t) e^{-2\pi i (\xi t +\eta t^d)} dt\,.
\end{equation}
Clearly $\mathfrak{m}_j(\xi, \eta)= \mathfrak{m}_d(2^{-j}\xi, 2^{-dj}\eta)$.
In $ T_{\Gamma,j, 1}$, the phase function $\phi_{\xi, \eta}(t)=
\xi t +\eta t^d $ does not have any critical point in a neighborhood of 
the support of $\rho$, and therefore a very rapid decay can be obtained by 
integration by parts so that we can show that $\sum_j{T_{\Gamma,j,1}}$ is essentially a finite sum of 
paraproducts (see Section \ref{paraproduct}).  A critical point of the phase function 
may occur in $ T_{\Gamma,j, 2}$ and therefore the method of stationary phase 
must be brought to bear in this case, exploiting in particular the oscillatory term.
This case requires the most extensive analysis. 

\begin{figure}  [ht]
\centering
\includegraphics[scale=.85]{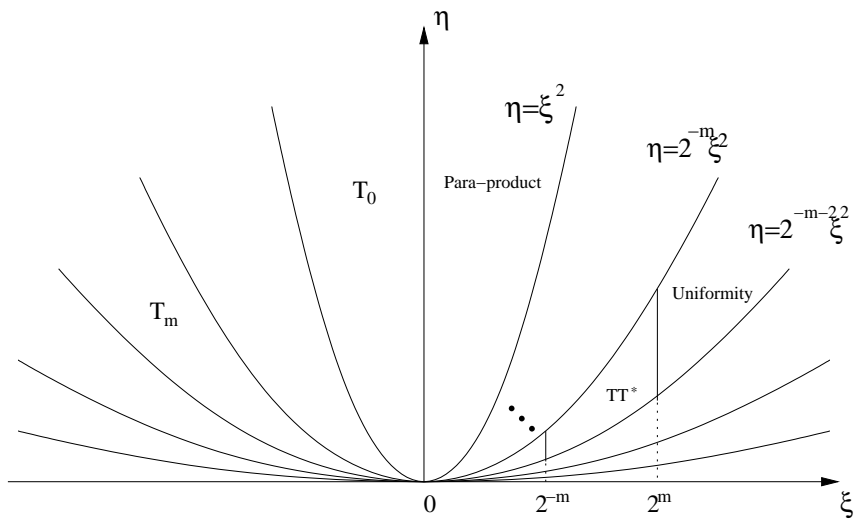}
\caption{The decomposition of $(\xi,\eta)$-plane for $\sum_{m}T_m$ when $d=2$}
\label{figure: Fig1}
\end{figure}

Notice that there are only finitely many $m'$ if $m$ is fixed in (\ref{defofTj2}). Without
loss of generality, we can assume $m'=m$. Then in order to get the $L^r$ estimates for 
$\sum_{j}T_{\Gamma, j,2}$, it suffices to prove the $L^r$ boundedness of 
$\sum_m T_m$,  where $T_m$'s are defined by
\begin{equation}\label{defofTm}
 T_m(f,g)(x)= \sum_{|j|>L}\iint {\wh f}(\xi){\wh g}(\eta)e^{2\pi
i(\xi+\eta)x}\wh\Phi\big(\frac{2^{-j}\xi}{2^m}\big)
\wh\Phi\big(\frac{2^{-d j}\eta}{2^{m}}\big)\mathfrak{m}_j(\xi,
\eta)d\xi d\eta\,.
\end{equation}
It can be proved that $T_0=\sum_{m\leq 0}T_m$ is equal to $\sum_{m\leq 0} O(2^{m/2})
 \Pi_{m}$, where $\Pi_m$ is a paraproduct studied in Theorem {\ref{para0uniest}}. 
This can be done by Fourier series and the cancellation condition of $\rho$ and thus 
$T_0$ is essentially a paraproduct.
We omit the details for it since it is exactly same as those in Section {\ref{paraproduct}}
for the case $\sum_j T_{\Gamma, j,1}$. Therefore, the most difficult term is $\sum_{m\geq 1}T_m$.
For this term, we have the following theorem. 

\begin{theorem}\label{thmofTm}
Let $T_m$ be a bilinear operator defined as in (\ref{defofTm}). Then 
there exists a constant $C$ such that
\begin{equation}\label{Tm221est}
\left\| \sum_{m\geq 1} T_m(f,g)\right\|_1 \leq  C\|f\|_2\|g\|_2
\end{equation}
holds for all $f, g\in L^2$.
\end{theorem}
 
A delicate analysis is required for proving this theorem. We will prove it 
in Subsection \ref{proofofthmTm}.  Theorem \ref{thmlargej} follows from 
Theorem \ref{thmofTm} and the boundedness of $\sum_j T_{\Gamma, j,1}$.
The rest of the article is organized as 
follows. In Section \ref{paraproduct}, the $L^r$-boundedness will be established 
for $\sum_j T_{\Gamma, j,1}$. Some crucial bilinear restriction estimates will
appear in Section \ref{BFR} and as a consequence Theorem \ref{thmofTm} follows.
Sections \ref{SphaseTri}-\ref{proofprop2neg} are devoted to a proof of the bilinear restriction 
estimates.

\section{Paraproducts and Uniform Estimates}\label{paraproduct}

In this section we prove that $\sum_j{T_{\Gamma,j,1}}$ is essentially a finite sum of 
certain paraproducts bounded from $L^p\times L^q$ to $L^r$.

First let us introduce the paraproduct encountered in our problem.   
Let $j\in \mathbb Z$,  $L_1, L_2$ be positive integers and $M_1,
M_2$ be integers.
$$ \w_{1,j}=[2^{L_1j+M_1}/2, 2\cdot2^{L_1 j+M_1}]$$
and
$$\w_{2,j}=[-2^{L_2 j+M_2}, 2^{L_2 j+M_2}]\,.$$
Let $\Phi_1$ be a Schwartz function whose Fourier transform is a
standard bump function supported on a small neighborhood of $[1/2, 2]$
or $[-2, -1/2]$, and $\Phi_2$ be a
Schwartz function whose Fourier transform is a standard bump function
supported on $[-1, 1]$ and $\wh\Phi_2(0)=1 $. For $\ell\in\{1,2\}$
and $n_1, n_2\in \mathbb Z$, define $\Phi_{\ell, j, n_\ell}$ by
$$
\wh\Phi_{\ell, j, n_\ell}(\xi)= \big(e^{2\pi i n_\ell (\cdot)} \wh\Phi_\ell (\cdot)\big) \bigg(\frac{\xi}{2^{L_\ell j +M_\ell}}\bigg)\,.
$$
It is clear that ${\wh\Phi_{\ell,j,n_\ell}}$ is supported on
$\w_{\ell,j}$. For locally integrable functions $f_\ell$'s, we
define $f_{\ell,j}$'s by
$$
f_{\ell,j, n_\ell}(x)=f_{\ell}*\Phi_{\ell, j, n_\ell}(x)\,.
$$
We now define a paraproduct to be
\begin{equation}\label{defofpara0}
\Pi_{L_1, L_2, M_1, M_2, n_1, n_2}(f_1, f_2)(x) = \sum_{j\in\mathbb Z}
 \prod_{\ell =1}^2 f_{\ell,j, n_\ell}(x) \,.
\end{equation}

For this paraproduct, we have the following uniform estimates. 

\begin{theorem}\label{para0uniest}
For any $p_1>1$, $p_2>1$ with $1/p_1+1/p_2=1/r$, there exists a
constant $C$ independent of $M_1, M_2, n_1, n_2$ such that
\begin{equation}\label{uniest1}
\big\|\Pi_{L_1, L_2, M_1, M_2, n_1, n_2}(f_1, f_2)\big\|_r \leq
  C\big(1+|n_1|\big)^{10}\big(1+|n_2|\big)^{10} \|f_1\|_{p_1}\|f_2\|_{p_2}\,,
\end{equation}
for all $f_1\in L^{p_1}$ and $f_2\in L^{p_2}$.
\end{theorem}

The $r>1$ case can be handled by a telescoping argument. The $r<1$ case 
is more complicated and it requires a time-frequency analysis. 
A proof of Theorem \ref{para0uniest} can be found in \cite{LiPara}. 
The constant $C$ in Theorem \ref{para0uniest} may depend on $L_1, L_2$. It is easy to see 
that $C$ is ${O}(\max\{2^{L_1}, 2^{L_2}\})$. It is
possible to get a much better upper bound such as $O\big(\log(1+
\max\{L_2/L_1, L_1/L_2\} )\big)$ by tracking the constants carefully in the proof in \cite{LiPara}. 
But we do not need the sharp constant in this article. 
The independence on $M_1, M_2$ is the most important issue here.\\

We now return to $\sum_j T_{\Gamma, j, 1}$.  This sum can be written as 
$ T_{L, 1} + T_{L,2} $, where $T_{L, 1}$ is a bilinear operator defined by
$$
\sum_{|j|>L}\sum_{m\in \mathbb Z}
\sum_{{\substack {m'\in \mathbb Z\\m'< m-10^d}}}
\int\int {\wh f}(\xi){\wh g}(\eta)e^{2\pi
i(\xi+\eta)x}\wh\Phi\big(\frac{2^{-j}\xi}{2^m}\big)
\wh\Phi\big(\frac{2^{-d j}\eta}{2^{m'}}\big) \mathfrak{m}_j(\xi,
\eta)d\xi d\eta\,,
$$
and $T_{L,2}$ is a bilinear operator given by 
$$
\sum_{|j|>L}\sum_{m'\in \mathbb Z}
\sum_{{\substack {m\in \mathbb Z\\ m < m'-10^d }}}
\int\int {\wh f}(\xi){\wh g}(\eta)e^{2\pi
i(\xi+\eta)x}\wh\Phi\big(\frac{2^{-j}\xi}{2^m}\big)
\wh\Phi\big(\frac{ 2^{-d j}\eta }{2^{m'}}\big) \mathfrak{m}_j(\xi,
\eta)d\xi d\eta\,.
$$

\underline{\it {Case $T_{L,1}$} }.  
We now prove that $T_{L,2}$ can be reduced to the paraproducts studied in \cite{LiPara}.
Indeed, if $|\xi| > 5^d |\eta|$, then let $\ti{\mathfrak m}$ be
$$
\ti{\mathfrak m}(\xi, \eta) = \mathfrak{m}_d(\xi, \eta) - \int\rho(t)e^{-2\pi i \xi t}dt\,.
$$
Clearly, if $|\xi|>5^d|\eta|$,
\begin{equation}\label{d1est}
 \big|D^\alpha \ti{\mathfrak m}(\xi, \eta)\big| \leq \frac{C_{\alpha, N}}{(1+\sqrt{|\xi|^2+|\eta|^2})^N}
\,\, {\rm for}\,\, {\rm all}\,\,N\in \{0\}\cup{\mathbb N}\,.
\end{equation}
and
\begin{equation}\label{m1est}
\big|\ti{\mathfrak m}(\xi, \eta)\big| \leq C_d \min\{1, |\eta|\}\,.
\end{equation}
Thus if $m'\leq m-10^d$,  then we have
$$
 \ti {\mathfrak m}(2^m\xi, 2^{m'}\eta) \wh\Phi(\xi)\wh\Phi(\eta) = \sum_{n_1,n_2\in\mathbb Z}
 C^{(1)}_{n_1, n_2} e^{2\pi i (n_1 \xi + n_2\eta)}\,,
$$
where $C^{(1)}_{n_1, n_2}$ is the Fourier coefficient. From
(\ref{d1est}) and (\ref{m1est}), we have
\begin{equation}\label{fouriercoe1}
|C^{(1)}_{n_1, n_2}|\leq \frac{C \min\{1,2^{m'/2}\}}{(1+ \sqrt{n_1^2+n_2^2})^{N} (1+2^m+2^{m'})^N}\,\, {\rm for}\,\, {\rm all}\,\,N\,.
\end{equation}
And notice that $\rho$ is an odd function, then we have
$$
\bigg(\int \rho(t)e^{-2\pi i 2^m\xi t}dt\bigg)\wh\Phi\big(\xi\big)
=\sum_n C^{(1)}_n e^{2\pi i n\xi}\,,
$$
where the Fourier coefficient $C^{(1)}_n$ satisfies
\begin{equation}\label{fconst1}
\big| C^{(1)}_n\big|\leq \frac{C_N\min\{2^{m/2},
1\}}{(1+|n|+2^m)^N}\,\,{\rm for}\,\, {\rm all}\,\, N\,.
\end{equation}
Thus $T_{L,1}$ equals to a sum of paraproducts
\begin{equation}\label{paraproduct1}
\sum_{j<-B}\sum_{m\in \mathbb Z}\sum_{{\substack {m'\in \mathbb Z\\
 m'\leq m-10^d }}}\!\!\!\! \bigg(\sum_{n_1, n_2\in \mathbb Z}C^{(1)}_{n_1, n_2}
 f_{j,m,n_1}(x)
g_{j,m',n_2}(x)  + \sum_{n\in \mathbb Z} C^{(1)}_n
f_{j,m,n}(x)g_{j,m',0}(x)\bigg)\,,
\end{equation}
where $f_{j,m,n}$ is a function whose Fourier transform is
$${\wh f_{j,m,n}}(\xi)= \wh f(\xi) \wh\Phi\big(2^{-j-m}\xi\big)
 e^{2\pi i n 2^{-j-m}\xi} \,,
$$
$g_{j,m',n}$ is a function whose Fourier transform is
$$
{\wh g_{j,m',n}}(\eta)= \wh g(\eta)
\wh\Phi\big(2^{-dj-m'}\eta\big)
 e^{2\pi i n 2^{-dj-m'}\eta}\,.
$$
\begin{remark}
Actually, in the definition of $f_{j,m,n}$ and $g_{j,m',n}$,
$\wh\Phi$ should be a Schwartz function supported in some
neighborhood of $\wh\Phi$ and it is identically equal to $1$
on the support of $\wh\Phi$. We abuse the notions here. But
it does no harm to us since the propety of the function does not
change significantly.
\end{remark}

Notice that the Fourier transform of $\sum_{
 m'< m-10^d}  g_{j,m',0}$ is a function supported
 in  the interval $I_{j,m,2}=[- 2^{dj+m}, 2^{dj+m} ]$. We denote
 $\sum_{m'<m-10^d}g_{j,m',0}$ by $g_{j,m}$. Thus by the definition of $\Phi$,
${\wh g_{j,m}} = \wh g
 \wh\Phi_m$ where $\Phi_m$ is a standard bump function supported in
 $I_{j,m,2}$. Let $I_{j,m,1}=[2^{j+m}/2, 2^{j+m+1}]$. Then ${\wh
 f_{j,m,n}}$ is supported in $I_{j,m,1}$. Hence due to the fast decay of the Fourier coefficients (\ref{fouriercoe1}) and (\ref{fconst1}), we actually run into two paraproducts in this case. One of them is
\begin{equation}\label{defpiB11}
\Pi^{(1)}_{L,1}(f, g)(x)= \sum_{j}f_{j,m,n_1}(x)g_{j,m',n_2}(x)\,.
\end{equation}
The other is
\begin{equation}\label{defofSb1}
\Pi^{(2)}_{L,1}(f,g)(x)=\sum_{j} f_{j,m,n}(x) g_{j,m}(x)\,.
\end{equation}

The $L^r$ estimates of these paraproducts follow from Theorem \ref{para0uniest}.
In fact, 
 for all $p,q>1$ and
 $1/p+1/q=1/r$, applying Theorem \ref{para0uniest}, we obtain  the $L^p\times L^q\rightarrow L^r$ estimates uniformly in  $m, m'$ with the operator 
norms $O((1+|n_1|+|n_2|)^{20})$ and
$O((1+|n|)^{10})$ for  $\Pi^{(1)}_{L,1}$ and
$\Pi^{(2)}_{L,1}$ respectively. The fast decay estimates of the Fourier coefficients
(\ref{fouriercoe1}) and (\ref{fconst1}) then allow us to conclude the
desired $L^r$ boundedness of $T_{L,1}$. \\

\underline{{\it Case $T_{L, 2} $}}.
This case is similar to the case $T_{L, 1}$. In fact, 
if $|\eta| > 5^d |\xi|$, then let $\ti {\mathfrak m}$ be
$$
\ti {\mathfrak m}(\xi, \eta) = {\mathfrak m}_d(\xi, \eta) -\int\rho(t)e^{-2\pi i \eta t^d}dt\,.
$$
Clearly, if $|\eta|>5^d|\xi|$,
\begin{equation}\label{d1estB12}
 \big|D^\alpha \ti {\mathfrak m}(\xi, \eta)\big| \leq \frac{C_{\alpha, N}}{(1+\sqrt{|\xi|^2+|\eta|^2})^N}
\,\, {\rm for}\,\, {\rm all}\,\,N\in \{0\}\cup{\mathbb N}\,.
\end{equation}
and
\begin{equation}\label{m1estB12}
\big|\ti {\mathfrak m}(\xi, \eta)\big| \leq C_d \min\{1, |\xi|\}\,.
\end{equation}
Thus if $m\leq m'-10^d$,  then we have
$$
 \ti {\mathfrak m}(2^m\xi, 2^{m'}\eta) \wh\Phi(\xi)\wh\Phi(\eta) = \sum_{n_1,n_2\in\mathbb Z}
 C^{(2)}_{n_1, n_2} e^{2\pi i (n_1 \xi + n_2\eta)}\,,
$$
where $C^{(2)}_{n_1, n_2}$ is the Fourier coefficient. From
(\ref{d1estB12}) and (\ref{m1estB12}), we have
\begin{equation}\label{fouriercoe1B12}
|C^{(2)}_{n_1, n_2}|\leq \frac{C \min\{1,2^{m/2}\}}{(1+ \sqrt{n_1^2+n_2^2})^{N} (1+2^m+2^{m'})^N}\,\, {\rm for}\,\, {\rm all}\,\,N\,.
\end{equation}
And notice that $\rho$ is an odd function, then we have
$$
\bigg(\int \rho(t)e^{-2\pi i 2^{m'}\eta t^d}dt\bigg)\wh\Phi\big(\xi\big)
=\sum_n C^{(2)}_n e^{2\pi i n\eta}\,,
$$
where the Fourier coefficient $C^{(2)}_n$ satisfies
\begin{equation}\label{fconst1B12}
\big| C^{(2)}_n\big|\leq \frac{C_N\min\{2^{m'/2},
1\}}{(1+|n|+2^{m'})^N}\,\,{\rm for}\,\, {\rm all}\,\, N\,.
\end{equation}
Thus $T_{L,2}$ equals to a sum of paraproducts
\begin{equation}\label{paraproduct1B12}
\sum_{j<-B}\sum_{m'\in \mathbb Z}\sum_{{\substack {m\in \mathbb Z\\
 m\leq m'-10^d }}} \bigg(\sum_{n_1, n_2\in \mathbb Z}C^{(2)}_{n_1, n_2}
 f_{j,m,n_1}(x)
g_{j,m',n_2}(x)  + \sum_{n\in \mathbb Z} C^{(2)}_n
f_{j,m,0}(x)g_{j,m',n}(x)\bigg)\,.
\end{equation}
Observe that the Fourier transform of $\sum_{
 m< m'-10^d}  f_{j,m,0}$ is a function supported
 in  the interval $I_{j,m',1}=[-2^{j+m'}, 2^{j+m'}]$. We denote
 $\sum_{m<m'-10^d}f_{j,m,0}$ by $f_{j,m'}$. Thus by the definition of $\Phi$,
${\wh f_{j,m'}} = \wh f
 \wh\Phi_{m'}$ where $\Phi_{m'}$ is a standard bump function supported in
 $I_{j,m',1}$. Let $I_{j,m',2}=[2^{j+m'}/2, 2^{j+m'+1}]$. Then ${\wh
 g_{j,m',n}}$ is supported in $I_{j,m',2}$. Hence due to the rapid decay of the Fourier coefficients (\ref{fouriercoe1B12}) and (\ref{fconst1B12}), we actually encounter two paraproducts in this case. One of them is
\begin{equation}\label{defpiB12}
\Pi^{(1)}_{L,2}(f, g)(x)= \sum_{j}f_{j,m,n_1}(x)g_{j,m',n_2}(x)\,.
\end{equation}
The other is
\begin{equation}\label{defofSb1B12}
\Pi^{(2)}_{L,2}(f,g)(x)=\sum_{j} f_{j,m'}(x) g_{j,m',n}(x)\,.
\end{equation}

Applying Theorem \ref{para0uniest}, for all $p,q>1$ and
 $1/p+1/q=1/r$,  we have
the $L^p\times L^q\rightarrow L^r$ estimates uniformly in  $m, m'$ with the operator norms 
$O((1+|n_1|+|n_2|)^{20})$ and
$O((1+|n|)^{10})$ for  $\Pi^{(1)}_{L,2}$ and
$\Pi^{(2)}_{L, 2}$ respectively. Then the
desired $L^r$ estimates of $T_{L, 2}$ follow due to
the fast decay of the Fourier coefficients (\ref{fouriercoe1B12}) and (\ref{fconst1B12}). \\

\section{Bilinear Fourier Restriction Estimates}\label{BFR}
\setcounter{equation}0

Let $d\geq 2, m\geq 0, j\in \mathbb Z $. We define a bilinear
Fourier restriction operator of $f, g$ by
\begin{equation}\label{defofbfr}
{\mathcal B}_{j, m} (f, g)(x)= 
2^{-(d-1)j/2}
\int_{\mathbb R} \!R_\Phi f(2^{-(d-1)j} x - 2^m t) R_\Phi g(x - 2^{m} t^d ) \rho(t) dt
 \,\,{\rm if} \,\, j\geq 0\,
\end{equation}
and
\begin{equation}\label{defofbfrjneg}
{\mathcal B}_{j, m} (f, g)(x)= 
2^{(d-1)j/2}
\int_{\mathbb R} \!R_\Phi f( x - 2^m t) R_\Phi g(2^{(d-1)j}x - 2^{m} t^d ) \rho(t) dt
 \,\,{\rm if} \,\, j< 0\,, 
\end{equation}
where $R_\Phi f$ and $R_\Phi g$ are the Fourier (smooth) restrictions of $
f, g$  on the support of $\wh\Phi$ respectively. More precisely, 
$R_\Phi f,  R_\Phi g$ are given by

\begin{equation}\label{defoff1}
 \wh {R_\Phi f  }(\xi)= \wh {f}(\xi) \wh\Phi(\xi)
\end{equation}
\begin{equation}\label{defofg1}
 \wh { R_\Phi g}(\xi)= \wh {g}(\xi) \wh\Phi(\xi)
\end{equation}

By inserting absolute values throughout and applying Cauchy-Schwarz
inequality, the boundedness of ${\mathcal B}_{j, m}$ from $L^2\times
L^2$ to $L^1$ follows immediately. Moreover, since the Fourier
transform of $f, g$ are restricted on the support of $\wh\Phi$, we
actually can improve the estimate. Let us state the improved
estimates by the following theorems, which are of independent
interest.

\begin{theorem}\label{thmBFR1m}
Let $d\geq 2$ and ${\mathcal B}_{j, m}$ be defined as in (\ref{defofbfr}) 
and (\ref{defofbfrjneg}).
If $L\leq |j| \leq m/(d-1)$,
then there exists a constant $C$ independent of $j, m$ such that
\begin{equation}\label{l221estBFR1m}
\big\|{\mathcal B}_{j, m} (f, g)\big\|_1\leq
 C 2^{\frac{(d-1)|j|-m}{8}} \|f\|_{2}\|g\|_{2}\,,
\end{equation}
holds for all $f, g\in L^2$.
\end{theorem}


\begin{theorem}\label{thmBFR1j}
Let $d\geq 2$ and ${\mathcal B}_{j, m}$ be defined as in (\ref{defofbfr})
and (\ref{defofbfrjneg}).
If $|j|\geq  m/(d-1)$,
then there exists a positive number $\e_0$ and
a constant $C$ independent of $j, m$ such that
\begin{equation}\label{l221estBFR1j}
\big\|{\mathcal B}_{j, m} (f, g)\big\|_1\leq
 C\max\left\{ 2^{\frac{m-(d-1)|j|}{3}}, 2^{-\e_0 m}\right\} \|f\|_{2}\|g\|_{2}\,,
\end{equation}
holds for all $f, g\in L^2$.
\end{theorem}


The positive number $\e_0$ in Theorem \ref{thmBFR1j} 
can be chosen to be $1/(8d)$.
Theorem \ref{thmBFR1m} 
 can be proved by a $TT^*$ method. 
However, the $TT^*$ method fails when $|j|>m/(d-1)$. To obtain Theorem \ref{thmBFR1j},
we will employ a method related to the uniformity of functions.  \\

Now we can see that Theorem \ref{thmofTm} is a consequence of Theorem \ref{thmBFR1m} and Theorem \ref{thmBFR1j}.

\subsection{Proof of Theorem \ref{thmofTm}}\label{proofofthmTm}
Define a bilinear operator $T_{j,m}$ to be 
\begin{equation}\label{defofTjm}
T_{j,m}(f,g)(x)= \iint {\wh f}(\xi){\wh g}(\eta)e^{2\pi
i(\xi+\eta)x}\wh\Phi\big(\frac{2^{-j}\xi}{2^m}\big)
\wh\Phi\big(\frac{2^{-d j}\eta}{2^{m}}\big)\mathfrak{m}_j(\xi,
\eta)d\xi d\eta\,.
\end{equation}
Let $\gamma_{j,m}$ be defined by
\begin{equation}\label{defofgamajm}
\gamma_{j,m}=\left\{ \begin{array}{ll}
  2^{\frac{(d-1)|j|-m}{8}} & {\rm if}\,\, |j|\leq\frac{m}{d-1}  \\
  \max\left\{ 2^{\frac{m-(d-1)|j|}{3}}, 2^{-\e_0 m} \right\}  & {\rm if}\,\, |j|\geq\frac{m}{d-1}
\end{array}
\right.
\end{equation}
A rescaling argument, Theorem  \ref{thmBFR1m} and Theorem \ref{thmBFR1j}
yield 
\begin{equation}\label{estTjm221}
 \left\| T_{j,m}(f,g) \right\|_1\leq C\gamma_{j,m}\|f\|_2\|g\|_2\,.
\end{equation}
Since $\sum_m T_m=\sum_{m}\sum_{j:|j|\geq L}T_{j,m}$, we obtain 
\begin{equation}\label{estTm221}
 \left\| \sum_{m\geq 1}T_{m}(f,g) \right\|_1
 \leq C\sum_{m\geq 1}\sum_{j:|j|\geq L} \gamma_{j,m}\left\|f_{j,m}\right\|_2
 \left\|g_{j,m}\right\|_2\,,
\end{equation}
where 
$$
 \wh{f_{j,m}}(\xi) =\wh f(\xi)\wh{\Phi}\left(\frac{\xi}{2^{j+m}}\right)\,,
$$
$$
 \wh{g_{j,m}}(\eta) =\wh{g}(\eta) \wh\Phi\left(\frac{\eta}{2^{dj+m}}\right)\,.
$$
Clearly the right hand side of (\ref{estTm221}) is bounded by
$C\|f\|_2\|g\|_2$. Therefore we finish the proof of Theorem {\ref{thmofTm}}.\\

We now start to make some reductions first for proving 
Theorem \ref{thmBFR1m} and Theorem \ref{thmBFR1j}.

\subsection{Smooth Truncations}

Let $\phi$ be a nonnegative Schwartz function such that $\wh\phi$ is
supported in $[-1/100, 1/100]$ and satisfies $\wh\phi(0)=1$. For any
integer $k$, let $\phi_k(x)=2^{-k}\phi(2^{-k}x)$. And for
$n\in\mathbb Z$, let
$$
 I_{k,n} = [2^{k}n, 2^{k+1}n]\,.
$$
Denote the characteristic function of the set $I$ by $ \Id_I$. We
define
$$
\Id^*_{k,n}(x)=\Id_{I_{k,n}}*\phi_k(x)\,,
$$
and
$$
\Id^{**}_{k,n}(x)= \int_{I_{k,n}}  \frac{2^{-k}}{\big(1+ 2^{-k}|x -y|\big)^{200}} dy\,.
$$
 Here $\Id^*_{k,n}$ and $\Id^{**}_{k,n} $  can be considered as essentially
$\Id_{I_{k,n}}$. Clearly we have
\begin{equation}\label{sum1kn}
\sum_n \Id^*_{k,n}(x)=1
\end{equation}

\begin{lemma}\label{restrictx1}
Let $\Phi_1$ be a Schwartz function such that $\wh\Phi_1(\xi)\equiv 1$
if $3/8\leq |\xi|\leq 17/8$ and $\wh\Phi_1$ is supported on $1/8\leq |\xi|\leq 19/8$.
And let $n \in \mathbb Z$ and define ${\mathcal {B}}_{j, m, n}$ by if $j>0$, then
\begin{equation}\label{defofBjmn}
{\mathcal {B}}_{j, m, n}(f, g)(x) =2^{-(d-1)j/2}\!\!
\int_{\mathbb R} \!R_{\Phi_1} f(2^{-(d-1)j} x - 2^m t) R_{\Phi_1} g(x - 2^{m} t^d ) \rho(t) dt
 \Id^*_{(d-1)j+m, n}(x)
\,,
\end{equation}
if $j\leq 0$, then 
\begin{equation}\label{defofBjmnneg}
{\mathcal {B}}_{j, m, n}(f, g)(x) =2^{(d-1)j/2}\!\!
\int_{\mathbb R} \!R_{\Phi_1} f( x - 2^m t) R_{\Phi_1} g(2^{(d-1)j}x - 2^{m} t^d ) \rho(t) dt
 \Id^*_{(d-1)|j|+m, n}(x)
\,,
\end{equation}

If there is a constant $C_{j, m}$ independent of $n$ such that
\begin{equation}\label{estofBjmn0} 
\|{\mathcal {B}}_{j, m, n}(f, g)\|_1\leq C_{j,m}\|f\|_2\|g\|_2
\end{equation}
 holds for all $f, g\in L^2$, then for any positive number $\e$, 
\begin{equation}\label{estofBjm0}
\|{\mathcal {B}}_{j, m}(f, g)\|_1\leq C_\e C_{j,m}^{1-\e} \|f\|_2\|g\|_2\,,
\end{equation} 
where $C_\e$ is a constant depending on $\e$ only.
\end{lemma}

\begin{proof}
Without loss of generality, assume that $C_{j,m}\leq 1$. And we only prove the case
$j>0$. The case $j\leq 0$ can be proved similarly and we omit the proof for this case.
From (\ref{sum1kn}), we can express $\langle{\mathcal {B}}_{j, m}(f, g), h\rangle$ as
$$
\sum_{k_1}\sum_{k_2}\sum_n \Lambda_{k_1, k_2, n, j, m}(f,g,h)\,.
$$
Here $\Lambda_{k_1, k_2, n, j, m}(f,g,h)$ equals to 
\begin{equation}
2^{-\frac{(d-1)j}{2}}
\int\!\!\int \! F_{1, j, m, n, k_1}(2^{-(d-1)j} x - 2^m t)
 F_{2, j, m, n, k_2}(x - 2^{m} t^d ) 
 F_{3, j, m, n}(x) dxdt\,, 
\end{equation}
where
$$F_{1,j, m,n, k_1}=\Id^*_{m, n+k_1} R_\Phi f\,, $$
$$F_{2, j, m, n, k_2}= \Id_{(d-1)j+m, n+k_2}^* R_\Phi g\,, $$ 
$$F_{3, j, m, n}=\Id^*_{(d-1)j+m, n}h\,. $$
Putting the absolute value throughout and utilizing the fast decay of $\Id^*_{k,n}$,
we estimate the sum of $\Lambda_{k_1, k_2, n, j, m}(f,g,h)$ for all $(k_1, k_2, n)$'s 
with $\max\{|k_1|, |k_2|\}\geq C_{j,m}^{-\e/2} $ by 
$$
2^{-\frac{(d-1)j}{2}}\!\!\!\!\!\!\!\!\sum_{\substack{(k_1, k_2, n)\\ \max\{|k_1|, |k_2|\}\geq C_{j,m}^{-\e/2} }}
\!\!\!\!\!\!\!\!
 \int\!\!\int \frac{ \big| R_\Phi f (2^{-(d-1)j} x - 2^m t)\big|  \big| R_\Phi g(x - 2^{m} t^d )\big| |h(x)||\rho(t)|}{ \big(1+ |t+k_1|\big)^N 
\big(1+| 2^{-(d-1)j} t^d +k_2 | \big)^N} dxdt\,,
$$
for all positive integers $N$. Since $|t|\sim 1$ when $t$ is in the support of $\rho$, we 
dominate this sum by 
\begin{equation}\label{k1k2large}
C_\e C_{j,m} \|f\|_2\|g\|_2\|h\|_\infty\,.
\end{equation}
We now turn to sum $ \Lambda_{k_1, k_2, n, j, m}(f,g,h) $ for all $|k_1|< C_{j,m}^{-\e/2}$ and $|k_2|< C^{-\e/2}_{j,m}$.  
Observe that when $j, m$ are large, Fourier transforms of $F_{1,j,m, n, k_1}$ and 
$F_{2,j, m, n, k_2}$ are supported in  $3/8\leq |\xi|\leq 17/8$. Thus we have 
$$
 \Lambda_{k_1, k_2, n, j, m}(f,g,h)= \langle {\mathcal B}_{j,m, n}(F_{1,j, m, n, k_1}, 
 F_{2,j, m, n, k_2}), h \rangle\,.
$$
And then (\ref{estofBjmn0}) gives
 $$
\sum_{\substack{(k_1, k_2, n)\\ \max\{|k_1|, |k_2|\} < C_{j,m}^{-\e/2} }} 
\!\!\!\!\!\!\!\big|\Lambda_{k_1, k_2, n, j, m}(f,g,h) 
\big| \leq C_{j,m}\!\!\!\!\!\!\!\sum_{\substack{(k_1, k_2, n)\\ \max\{|k_1|, |k_2|\} < C_{j,m}^{-\e/2} }}
\!\!\!\!\!\!\!
 \| F_{1,j, m, n, k_1}\|_2\|F_{2,j, m, n, k_2} \|_2\|h\|_\infty\,,
$$
which is clearly bounded by
\begin{equation}\label{k1k2small}
 C_{j, m}^{1-\e} \|f\|_2\|g\|_2\|h\|_\infty\,.
\end{equation}
Combining (\ref{k1k2large}) and (\ref{k1k2small}), we complete the proof.
\end{proof}

\subsection{Trilinear Forms}
Let $f_1, f_2, f_3$ be measurable functions supported on $1/16\leq |\xi|\leq 39/16$.
Define a trilinear form $\Lambda_{j,m, n}(f_1, f_2, f_3)$ by
\begin{equation}\label{defofLajmn00}
\Lambda_{j,m, n}(f_1, f_2, f_3):= \left\langle 
 {\mathcal B}_{j,m,n}(\check {f_1}, \check {f_2} ), \check {f_3} 
\right\rangle\,.
\end{equation}
By Lemma {\ref{restrictx1}}, Theorem \ref{thmBFR1m} and Theorem \ref{thmBFR1j} can be reduced to 
the following theorems respectively.

\begin{theorem}\label{thmtriest1m}
Let $d\geq 2$ and $\Lambda_{j,m, n}(f_1, f_2, f_3)$ be defined as in (\ref{defofLajmn00}).
If $ |j| \leq m/(d-1)$,
then there exists 
a constant $C$ independent of $j, m$ such that
\begin{equation}\label{l221triest1m}
\left |  \Lambda_{j,m, n}(f_1, f_2, f_3)     \right|\leq
 C 2^{\frac{-(d-1)|j|-m}{2}}2^{-\frac{m-(d-1)|j|}{6}} \|f_1\|_{2}\|f_2\|_{2} \|{f_3} \|_2\,,
\end{equation}
holds for all $f_1, f_2, f_3 \in L^2$.
\end{theorem}

\begin{theorem}\label{thmtriest1}
Let $d\geq 2$ and $\Lambda_{j,m, n}(f_1, f_2, f_3)$ be defined as in (\ref{defofLajmn00}).
If $|j|\geq  m/(d-1)$,
then there exist a positive number $\e_0$ and
a constant $C$ independent of $j, m$ such that
\begin{equation}\label{l221triest1}
\left |  \Lambda_{j,m, n}(f_1, f_2, f_3)     \right|\leq
 C\max\left\{ 2^{\frac{-(d-1)|j|+m}{2}}, 2^{-\e_0 m}\right\} \|f_1\|_{2}\|f_2\|_{2} \|\wh {f_3} \|_\infty\,,
\end{equation}
holds for all $f_1, f_2 \in L^2$ and $\wh {f_3}\in L^\infty$ such that 
$f_1, f_2, f_3$ are supported on $1/16\leq |\xi|\leq 39/16$.
\end{theorem}

A proof of Theorem \ref{thmtriest1m} will be provided in Section \ref{SphaseTri} and
a proof of Theorem \ref{thmtriest1} will be given in Section \ref{triest000neg}.

\section{Stationary Phases and Trilinear Oscillatory Integrals}\label{SphaseTri}
\setcounter{equation}0

In Section \ref{paraproduct}, we see that Fourier series can help
us to reduce the problem to the paraproduct case when $|m'-m|>10^d$.
This method does not work for the case when $|m-m'|\leq 10^d$. This is because
the critical points of the phase function may
happen in a neighborhood of $1/2\leq |t|\leq 2$, say $1/4\leq |t|\leq 5/2$,
which provides a stationary phase for the Fourier integral $\mathfrak{m}_d$.
This stationary phase gives a highly oscillatory factor in the integral. 
We expect a suitable decay from the highly oscillatory factor.  In this 
section we should prove Theorem \ref{thmtriest1m} by utilizing essentially a $TT^*$
method.  \\

Let $\Lambda_{j,m}(f_1, f_2, f_3)= \langle \mathcal{B}_{j,m}(\check{f_1}, \check{f_2}), \check{f_3} \rangle
$.
To prove Theorem \ref{thmtriest1m}, it suffices to prove the following $L^2$ estimate 
for the trilinear form $\Lambda_{j,m}(f_1, f_2, f_3)$,
\begin{equation}\label{estLajm222}
\left|\Lambda_{j,m}(f_1, f_2, f_3)  \right|\leq C2^{\frac{-(d-1)|j|-m}{2}}
2^{- \frac{m-(d-1)|j|}{6}} \|f_1\|_{2}\|f_2\|_{2} \|{f_3} \|_2\,,
\end{equation}
holds for all $f_1, f_2, f_3 \in L^2$. Clearly $\Lambda_{j,m}(f_1, f_2, f_3)$ can be 
expressed as if $j>0$, 
$$
 2^{-\frac{(d-1)j}{2}} \iint f_1(\xi)\wh\Phi(\xi) f_2(\eta)\wh\Phi(\eta) f_3\left(2^{-(d-1)j}\xi+\eta\right) 
    \mathfrak{m}_d\left(2^{m}\xi, 2^{m}\eta\right) d\xi d\eta\,, 
$$
and if $j\leq 0$, 
$$
 2^{\frac{(d-1)j}{2}} \iint f_1(\xi)\wh\Phi(\xi) f_2(\eta)\wh\Phi(\eta) f_3\left(\xi+2^{(d-1)j}\eta\right) 
    \mathfrak{m}_d\left(2^{m}\xi, 2^{m}\eta\right) d\xi d\eta\,, 
$$

Whenever $\xi, \eta\in {\operatorname {supp}}\wh\Phi$, the second order derivative 
of the phase function $\phi_{m,\xi, \eta}(t)=2^m(\xi t+\eta t^d)$ is comparable to 
$2^m$. We only need to focus on the worst situation when there is a critical point 
of the phase function in a small neighborhood of $\operatorname{supp}\rho$. Thus 
the method of stationary phase yields
\begin{equation}\label{prinpartofmd}
\mathfrak{m}_d\left(2^{m}\xi, 2^{m}\eta\right)\sim 2^{-m/2} e^{ i c_d 2^m \xi^{d/(d-1)}\eta^{-1/(d-1)}}\,,
\end{equation}   
where $c_d$ is a constant depending only on $d$.  Henceforth we reduce Theorem \ref{thmtriest1m}
to the following lemma.

\begin{proposition}\label{triosc*222}
Let $\Lambda^*_{j,m}$ be defined by if $j>0$ then
\begin{equation}\label{defofLa*jm}
\Lambda^*_{j,m}(f_1, f_2, f_3) = \iint f_1(\xi)\wh\Phi(\xi) f_2(\eta)\wh\Phi(\eta) f_3\left(2^{-(d-1)j}\xi+\eta\right) 
 e^{ i c_d 2^m \xi^{d/(d-1)}\eta^{-1/(d-1)}} d\xi d\eta\,,
\end{equation}
and if $j\leq 0$, then 
\begin{equation}\label{defofLa*jmneg}
\Lambda^*_{j,m}(f_1, f_2, f_3) = \iint f_1(\xi)\wh\Phi(\xi) f_2(\eta)\wh\Phi(\eta) f_3
\left(\xi+2^{(d-1)j}\eta\right) e^{ i c_d 2^m \xi^{d/(d-1)}\eta^{-1/(d-1)}} d\xi d\eta\,,
\end{equation}
Then there exists a positive constant $C$ such that 
\begin{equation}\label{estLa*jm222}
\left|\Lambda^*_{j,m}(f_1, f_2, f_3)  \right|\leq C2^{-\frac{m-(d-1)|j|}{6}} \|f_1\|_{2}\|f_2\|_{2} \|{f_3} \|_2\,,
\end{equation}
holds for all $f_1, f_2, f_3 \in L^2$.
\end{proposition}

\begin{proof}
Without loss of generality, we assume that $\wh\Phi$ is supported on $[1/2, 2]$ (or
$[-2, -1/2]$). And we only give a proof for the case $j>0$ since a similar argument 
yields the case $j\leq 0$.  
Let $\phi_{d,m}$ be a phase function defined by 
$$
 \phi_{d,m}(\xi, \eta) = c_d \xi^{d/(d-1)}\eta^{-1/(d-1)}\,.
$$ 
And let $b_1=1-2^{-(d-1)j}$ and $b_2=2^{-(d-1)j}$. 
Changing variable 
$\xi\mapsto \xi-\eta$ and $\eta\mapsto b_1\xi+b_2\eta$, 
we have that $\Lambda^*_{j, m}(f_1,f_2,f_3)$ equals
$$
\iint \!\! f_1(\xi-\eta) f_2(b_1\xi+b_2\eta) f_3(\xi)
\wh\Phi(\xi-\eta)\wh\Phi(b_1\xi+b_2\eta) e^{i2^{m}
\phi_{d,m}(\xi-\eta, b_1\xi+b_2\eta ) } d\xi d\eta.
$$
Thus by Cauchy-Schwarz we dominate $|\La^*_{j, m}(f_1, f_2, f_3)|$
by
$$
 \big\| \bT_{d, j, m}(f_1, f_2)\big\|_2\|f_3\|_2\,,
$$
where $\bT_{d, j, m}$ is defined by
$$
\bT_{d, j,  m}(f_1, f_2)(\xi)\!
=\!\!\! \int\!\! f_1(\xi-\eta) f_2(b_1\xi+b_2\eta) 
\wh\Phi(\xi-\eta)\wh\Phi(b_1\xi+b_2\eta) e^{i2^{m}
\phi_{d,m}(\xi-\eta, b_1\xi+b_2\eta )} d\eta.
$$ 
It is easy to see that
 $\big\| \bT_{d,j,  m}(f_1, f_2)\big\|_2^2$ equals to
$$
\int\bigg(\iint F(\xi, \eta_1, \eta_2) G(\xi, \eta_1, \eta_2)
e^{ i 2^{m} \big(\phi_{d,m}(\xi-\eta_1, b_1\xi+b_2\eta_1 )-
\phi_{d,m}(\xi-\eta_2, b_1\xi+b_2\eta_2) \big)}
  d\eta_1d\eta_2 \bigg)  d\xi\,, 
$$
where  
$$
F(\xi, \eta_1, \eta_2)= \big(f_1\wh\Phi\big)(\xi-\eta_1 )
\overline{\big(f_1\wh\Phi\big)(\xi-\eta_2)  }
$$
$$
G(\xi, \eta_1,\eta_2)=\big(f_2\wh\Phi\big)(b_1\xi+b_2\eta_1 )
 \overline{\big(f_2\wh\Phi\big)(b_1\xi+b_2\eta_2)}\,.
$$
Changing variables $\eta_1\mapsto \eta$ and $\eta_2\mapsto \eta+\tau$, 
we see that $\big\|\bT_{d,j, m}(f_1, f_2) \big\|_2^2$ equals to
$$
\int\!\!\bigg(\!\int\!\!\!\int \!F_\tau(\xi-\eta) G_{\tau}(b_1\xi+b_2\eta)
 e^{i 2^{m}\!\big(\phi_{d,m}(\xi-\eta, b_1\xi+b_2\eta)-
\phi_{d,m}\left(\xi-\eta-\tau, b_1\xi+b_2(\eta+\tau)\right )
\big)}  \!d\xi d\eta                           \!\!  \bigg)d\tau , 
$$
where 
$$
F_\tau(\cdot) = \big(f_1\wh\Phi\big)(\cdot )
\overline{\big(f_1\wh\Phi\big)(\cdot -\tau )  }
$$
$$
G_\tau(\cdot) = \big(f_2\wh\Phi\big)(\cdot )
\overline{\big(f_2\wh\Phi\big)(\cdot + b_2\tau )  }\,.
$$
Changing coordinates to $(u,v)=(\xi-\eta, b_1\xi+b_2\eta)$, 
the inner integral becomes
\begin{equation}\label{inner}
\int\!\int F_\tau(u)G_\tau(v) e^{i2^{m}
 \ti\bQ_\tau (u, v)  } du dv\,,
\end{equation}
where $\ti\bQ_{\tau}$ is defined by 
$$
\ti\bQ_\tau(u, v)=\phi_{d,m}(u, v)-\phi_{d,m}(u-\tau, v+b_2\tau)\,.
$$
When $j$ is large enough, the mean value theorem yields 
\begin{equation}\label{mixdlarge111}
 \left|\partial_u\partial_v \ti\bQ_\tau(u, v)\right|\geq C\tau\,,
\end{equation}
if $u, v, u-\tau, u+b_2\tau\in{\rm supp}\wh\Phi$.

A well-known H\"ormander theorem on the non-degenerate phase  \cite{Hor, PSt1} 
gives that (\ref{inner}) is estimated  by 
$$
C \min\big\{1, 2^{-m/2}|\tau|^{-1/2 } \big\} 
\big\|F_\tau\big\|_2 \big\|G_\tau \big\|_2 \,.
$$
Hence by Cauchy-Schwarz inequality $\big\| \bT_{d, j,  m}(f_1, f_2)\big\|_2^2$ is bounded by
$$
 \tau_0 \|f_1\|_2^2\|f_2\|_2^2 + C\int_{ \tau_0<|\tau|<10} 
  \min\big\{1, 2^{-m/2}|\tau|^{-1/2 } \big\} 
\big\|F_\tau\big\|_2 \big\|G_\tau \big\|_2 d\tau\,
$$
for any $\tau_0>0$. By one more use of Cauchy-Schwarz inequality,
$\big\| \bT_{d, j,  m}(f_1, f_2)\big\|_2^2$ is dominated by
$$
 \left(\tau_0 +
 C \tau_0^{-1/2} 2^{-m/2}2^{(d-1)j/2}\right)\|f_1\|_2^2\|f_2\|_2^2\,,
$$
for any $\tau_0>0$. 
 Thus we have 
\begin{equation}\label{La*jmest422}
 \big| \La_{j,  m}^*(f_1, f_2, f_3)\big|\leq 
C 2^{\frac{(d-1)j-m}{6} } \|f_1\|_2\|f_2\|_2\|f_3\|_2\,.
\end{equation}
This completes the proof of Proposition  \ref{triosc*222}.

\end{proof}

It is easy to see that 
\begin{equation}\label{La*jmest4221}
\big| \La_{j,  m}^*(f_1, f_2, f_3)\big|\leq 
C 2^{-\e m} \|f_1\|_2\|f_2\|_2\|f_3\|_2
\end{equation}
fails for all  $|j|\geq m/(d-1)$. Indeed, 
let us only consider the case $j> m/(d-1)$. Assume that (\ref{La*jmest4221})
holds for all $j> m/(d-1)$. Let $j\rightarrow \infty$, then (\ref{La*jmest4221})
implies 
\begin{equation}
\big| \La_{m}^*(f_1, f_2, f_3)\big|\leq 
C 2^{-\e m} \|f_1\|_2\|f_2\|_2\|f_3\|_2\,,
\end{equation}
where 
$$
\La_{m}^*(f_1, f_2, f_3)=\iint f_1(\xi)\wh\Phi(\xi) f_2(\eta)\wh\Phi(\eta) f_3\left(\eta\right) 
 e^{ i c_d 2^m \xi^{d/(d-1)}\eta^{-1/(d-1)}} d\xi d\eta\,.
$$
Simply taking $f_2=f_3$, we obtain
\begin{equation}
 \sup_{\eta\sim 1}\left| \int f_1(\xi)\wh\Phi(\xi) e^{i c_d 2^m \xi^{d/(d-1)}\eta^{-1/(d-1)}}d\xi  \right|
\leq C 2^{-\e m}\|f_1\|_2\,.
\end{equation}
This clearly can not be true and hence we get a contradiction. Therefore, 
(\ref{La*jmest4221}) does not hold for all $j>m/(d-1)$. From this fact, we know 
that the $TT^*$ method can not work for the case $|j|>m/(d-1)$. In 
the following sections, we have to
introduce a concept of uniformity and employ a "quadratic" Fourier analysis.

\section{Uniformity}\label{unif}
\setcounter{equation}0

We introduce a concept related to a notion of uniformity employed by
Gowers \cite{G}. A similar uniformity was utilized in \cite{CLTT}.
Let $\sigma\in (0, 1]$, let $\mQ$ be a collection of some
real-valued measurable functions, and fix a bounded interval $\bI$
in $\ZR $.

\begin{definition}\label{defofunif}
A function $f\in L^2(\bI)$ is $\sigma$-uniform in $\mQ$
 if
\begin{equation}\label{uniformity0}
 \bigg|\int_\bI f(\xi) e^{-i q(\xi)} d\xi\bigg|\leq \sigma\|f\|_{L^2(\bI)}\,
\end{equation}
for all $q\in {\mathcal Q}$. Otherwise, $f$ is said to be
$\sigma$-nonuniform in $\mQ$.
\end{definition}

\begin{theorem}\label{uniformitythm}
Let $L$ be a bounded sub-linear functional from $L^2(\bI)$ to ${\mathbb
C}$, let $\bS_\sigma$ be the set of all functions that are
$\sigma$-uniform in $\mQ$, and let
\begin{equation}\label{unifnorm1}
U_\sigma = \sup_{f\in\bS_\sigma}\frac{|L(f)|}{\|f\|_{L^2(\bI)}}\,.
\end{equation}
Then for all functions in $L^2(\bI)$,
\begin{equation}\label{Lnonunif}
 \big|L(f)\big|\leq \max\big\{U_\sigma,  2\sigma^{-1}
 Q\big\}\|f\|_{L^2(\bI)}\,\,,
\end{equation}
 where
\begin{equation}\label{defofB}
 Q = \sup_{q\in \mQ}\big|L(e^{iq})\big|\,.
\end{equation}
\end{theorem}

\begin{proof}
Clearly the complement $\bS_\sigma^c$ is a set of all functions that are
$\sigma$-nonuniform in $\mQ$.  Let us set
$$
A:= \sup_{f\in L^2(\bI)}\frac{|L(f)|}{\|f\|_{L^2(\bI)}}\,\,\,\,{\rm
and} \,\,\,\, A_1:=
\sup_{f\in\bS^c_\sigma}\frac{|L(f)|}{\|f\|_{L^2(\bI)}}\,\,.
$$
Clearly $A=\max\{ A_1, U_\sigma\}$. In order to obtain (\ref{Lnonunif}), it
suffices to prove that if $U_\sigma < A_1$, then
\begin{equation}\label{nonunif11}
A_1\leq 2\sigma^{-1}Q\,.
\end{equation}

For any $\e>0$, there exists a function $f\in \bS^c_\sigma$ such
that
\begin{equation}\label{supLf}
(A_1-\e)\|f\|_{L^2(\bI)} \leq |L(f)|\,.
\end{equation}

Let $\langle \cdot, \cdot\rangle_\bI$ be an inner product on
$L^2(\bI)$ defined by
$$
\langle f, g\rangle_\bI =\int_\bI f(x){\overline {g(x)} }dx\,,
$$
for all $f, g\in L^2(\bI)$.  Since $f$ is $\sigma$-nonuniform in
$\mQ$, there exists a function $q$ in $\mQ$ such that
\begin{equation}\label{nonunif1}
\big| \langle f, e^{iq}\rangle_\bI\big| \geq \sigma
\|f\|_{L^2(\bI)}\,.
\end{equation}
Let $g\in L^2(\bI)$ such that $g\perp e^{iq}$ and
$\|g\|_{L^2(\bI)}=1$. Then we can write $f$ as
\begin{equation}\label{ftogeq}
 f = \langle f, g\rangle_\bI \,g + \frac{\langle f,
 e^{iq}\rangle_\bI}{|\bI|}e^{iq}\,.
\end{equation}
Sub-linearity of $L$ and the triangle inequality then yield
\begin{equation}\label{Lftogq}
\big|L(f)\big| \leq  \big|\langle f, g\rangle_\bI\big| \big|
L(g)\big|
   + |\bI|^{-1}\big|\langle f, e^{iq}\rangle_\bI\big|
   \big|L(e^{iq})\big|\,.
\end{equation}
Notice that $A=A_1$ if $U_\sigma < A_1$ and
\begin{equation}\label{fgest}
 \langle f, f\rangle_\bI= \big|\langle f, g\rangle_\bI\big|^2
   + |\bI|^{-1}\big|\langle f, e^{iq}\rangle_\bI\big|^2\,.
\end{equation}
Then from (\ref{supLf}) and (\ref{Lftogq}), we have
\begin{equation}\label{A1est00}
 (A_1-\e)\|f\|_{L^2(\bI)}\leq A_1\|f\|_{L^2(\bI)}\sqrt{1- \frac{\big|\langle f, e^{iq}\rangle_\bI\big|^2 }{|\bI|\langle f, f\rangle_\bI}}
  +|\bI|^{-1}\big|\langle f, e^{iq}\rangle_\bI\big| Q\,.
\end{equation}
 Applying the elementary inequality $ \sqrt{1-x} \leq 1-x/2 $ if
$0\leq x\leq 1$, we then get
\begin{equation}\label{A1est01}
 A_1 \leq \frac{2\|f\|_{L^2(\bI)}}{\big|\langle f,
 e^{iq}\rangle_\bI\big|}Q  + \e |I|\frac{2\|f\|_{L^2(\bI)}^2 }{\big|\langle f,
 e^{iq}\rangle_\bI\big|^2}\,.
\end{equation}
From (\ref{nonunif1}), we have
\begin{equation}\label{A1est02}
 A_1 \leq 2\sigma^{-1}Q  + 2\e |I| \sigma^{-2}\,.
\end{equation}
Now let $\e\rightarrow 0$ and we then obtain (\ref{nonunif11}).
Therefore we complete the proof.

\end{proof}

\section{Estimates of the trilinear forms, Case $j>0$}\label{triest000}
\setcounter{equation}0

Without loss of generality, in the following sections we assume that
$f_i$ is supported on $\bI_{i}$
for $i\in \{1,2,3\}$, where $\bI_i$ is 
either $[1/16, 39/16] $ or $[-39/16, -1/16] $. 
Let $\mQ_1$ be a set of some functions defined by
\begin{equation}\label{defofmQ1}
 \mQ_1= \left\{ a\xi^{{d}/{d-1}} + b\xi :  2^{m-100} \leq |a|\leq
 2^{m+100}  \,\,{\rm and}\,\, a, b\in\ZR
     \right\}\,.
\end{equation}

\begin{proposition}\label{prop1}
Let $f_1$ be  $\sigma$-uniform in $\mQ_1$. 
And let $j>0$ and  $ \Lambda_{j,m, n}(f_1, f_2, f_3)  $ be defined as in (\ref{defofLajmn00}).
Then 
 there exists a constant $C$ independent of $j, m, n,  f_1$ such that
\begin{equation}\label{f1unifest}
\left |  \Lambda_{j,m, n}(f_1, f_2, f_3) \right|\leq
 C 2^{-\frac{(d-1)j}{2}-\frac{m}{2}} \max\left\{2^{-100m},  2^{\frac{-(d-1)j+m}{2}}, \sigma  \right\} 
\prod_{i=1}^3\|f_i\|_{L^2(\bI_i)}\,,
\end{equation} 
holds for all $f_2\in L^2(\bI_2)$ and $f_3\in L^2(\bI_3)$. 
\end{proposition}

\begin{proof}
Let $\Id_{m,l}=\Id_{\bI_{m,l}}$ and let ${\mathcal B}_{j,m,n, \ell}$ be a bilinear operator defined by
$$
 {\mathcal B}_{j,m,n, \ell}(f,g)(x) = {\mathcal B}_{j,m,n}(f,g)(x) \Id_{m,\ell}(x)\,,
$$
for all $f, g$.
Decompose $\Lambda_{j,m, n}(f_1, f_2, f_3)$ into
$\sum_{\ell} \Lambda_{j,m,n, \ell}$,
where 
$$ 
\Lambda_{j,m,n, \ell}(f_1, f_2, f_3) = 
\left\langle  {\mathcal B}_{j,m,n, \ell}(\check {f_1}, \check{f_2}), 
  \check {f_3} \right\rangle \,.
$$
Let $\alpha_{m, \ell}$ be a fixed point in the interval $\bI_{m, \ell}$.
And set $F_{\Phi_1,j,m,\ell}(x, t)$ to be 
$$
 F_{\Phi_1,j,m,\ell}(x, t):=  R_{\Phi_1} \check {f_1} (2^{-(d-1)j} x - 2^m t)
     - R_{\Phi_1} \check{f_1}(2^{-(d-1)j}\alpha_{m, \ell} - 2^m t) 
$$

Split ${\mathcal B}_{j,m,n, \ell}(\check{f_1}, \check{f_2})$ into two terms:
$$
 {\mathcal B}_{j,m,n, \ell}^{(1)}\left(\check{f_1}, \check{f_2}\right) + 
{\mathcal B}_{j,m,n, \ell}^{(2)} \left(\check{f_1}, \check{f_2}\right)\,,
$$
where ${\mathcal B}_{j,m,n, \ell}^{(1)}\left(\check{f_1}, \check{f_2}\right)$ is equal to
$$
2^{-(d-1)j/2}\!\!
\int_{\mathbb R} \!F_{\Phi_1,j,m,\ell}(x, t)
 R_{\Phi_1} \check{f_2}(x - 2^{m} t^d ) \rho(t) dt
 \left( \Id^*_{(d-1)j+m, n}(x)  \Id_{m, \ell}(x) \right)
$$
and ${\mathcal B}_{j,m,n, \ell}^{(2)} \left(\check{f_1}, \check{f_2}\right)$ equals to
$$
2^{-(d-1)j/2}\!\!
\int_{\mathbb R} \!
     R_{\Phi_1} {\check{f_1}}(2^{-(d-1)j}\alpha_{m, \ell} - 2^m t) 
 R_{\Phi_1} \check{f_2}(x - 2^{m} t^d ) \rho(t) dt
 \left( \Id^*_{(d-1)j+m, n}(x) \Id_{m, \ell}(x)\right)\,.
$$
For $i=1, 2$, let $ \Lambda_{j,m, n}^{(i)}(f_1, f_2, f_3)$ denote
$$
\sum_\ell  \left\langle {\mathcal B}_{j,m,n, \ell}^{(i)}\left(\check{f_1}, \check{f_2}\right), 
  \check{f_3} \right\rangle  \,.
$$
We now start to prove that 
\begin{equation}\label{estB1jmnl1}
\left|\Lambda_{j,m, n}^{(1)}(f_1, f_2, f_3) \right|
\leq  2^{-\frac{(d-1)j}{2} } 2^{-(d-1)j+m} \left\| \check{f_1}\right\|_\infty
           \left\|\check{f_2}\right\|_2 \left\|\check{f_3}\right\|_2\,.
\end{equation} 
The mean value theorem and the smoothness of $\Phi_1$ yield 
that for $x\in\bI_{m,\ell}$, 
\begin{equation}\label{estdiffFjml1}
 \left| F_{\Phi_1,j,m,\ell}(x, t) \right|\leq C\left\| \check{f_1}\right\|_\infty
  2^{-(d-1)j} \left| x-\alpha_{m, \ell}\right| \leq  C2^{-(d-1)j+m}\left \| \check{f_1}\right\|_\infty \,.              
\end{equation} 
Because $|t|\sim 1$ when $t\in\operatorname{supp}\rho$, 
${\mathcal B}_{j,m,n, \ell}^{(1)}\left(\check{f_1}, \check{f_2}\right)$
can be written as
\begin{equation}\label{sumlk1111}
2^{-\frac{(d-1)j}{2}}\!\!
\int_{\mathbb R} \!F_{\Phi_1,j,m,\ell}(x, t)
 \sum_{\ell_0}\left(\Id_{m, \ell + \ell_0 } 
R_{\Phi_1} \check{f_2}\right) (x - 2^{m} t^d ) \rho(t) dt
 \left( \Id^*_{(d-1)j+m, n}(x)  \Id_{m, \ell}(x) \right)\,,
\end{equation}
where $\ell_0$ is an integer between $-10$ and $10$.
Putting absolute value throughout and applying (\ref{estdiffFjml1})
plus 
 Cauchy-Schwarz inequality, we then estimate 
$ \left|\Lambda_{j,m, n}^{(1)}(f_1, f_2, f_3) \right|$ by
$$
 C2^{-\frac{(d-1)j}{2}}2^{-(d-1)j+m}\left\| \check{f_1}\right\|_\infty 
\sum_{\ell_0=-10}^{10} 
\sum_\ell\left\|\Id_{m, \ell + \ell_0 } 
R_{\Phi_1} \check{f_2}  \right\|_2 
\left\| \Id_{m, \ell}\check{f_3}\right\|_2\,,
$$
which clearly gives (\ref{estB1jmnl1}) by one more use of Cauchy-Schwarz inequality.\\

We now prove that 
\begin{equation}\label{estB1jmnl122}
\left|\Lambda_{j,m, n}^{(1)}(f_1, f_2, f_3) \right|
\leq  2^{-\frac{(d-1)j}{2} } 2^{-m} \left\| \check{f_1}\right\|_1
           \left\|\check{f_2}\right\|_2 \left\|\check{f_3}\right\|_2\,.
\end{equation} 
From (\ref{sumlk1111}), we get that $\Lambda_{j,m, n}^{(1)}(f_1, f_2, f_3)$ equals to
$$
2^{-\frac{(d-1)j}{2}}\sum_{\ell_0=-10}^{10}
\sum_\ell \Lambda_{j,m, n,\ell_0, \ell,1}(f_1, f_2, f_3) -  \Lambda_{j,m, n,\ell_0, \ell,2}(f_1, f_2, f_3)\,,
$$
where $\Lambda_{j,m, n,\ell_0, \ell,1}(f_1, f_2, f_3)$ is equal to
$$
\int_{{\mathbb R}^2} \!\! R_{\Phi_1} \check {f_1} (2^{-(d-1)j} x \!- \!2^m t)
\!\left(\Id_{m, \ell + \ell_0 } 
R_{\Phi_1} \check{f_2}\right) \!(x \!- \!2^{m} t^d ) \rho(t)\! 
 \left( \Id^*_{(d-1)j+m, n} \Id_{m, \ell} \check{f_3} \right)\!(x)  dt dx
$$
and $\Lambda_{j,m, n,\ell_0, \ell,2}(f_1, f_2, f_3)$  equals to
$$
\int_{{\mathbb R}^2} \!\! R_{\Phi_1} \check {f_1} (2^{-(d-1)j} \alpha_{m, \ell} - \!2^m t)\!
\left(\Id_{m, \ell + \ell_0} 
R_{\Phi_1} \check{f_2}\right) \!(x - 2^{m} t^d ) \rho(t)\! 
 \left( \Id^*_{(d-1)j+m, n} \Id_{m, \ell} \check{f_3} \right)\!(x) dt dx .
$$
 Cauchy-Schwarz inequality 
yields that 
\begin{equation}\label{estLajmnkl2}
\left|\Lambda_{j,m, n,\ell_0, \ell,2}(f_1, f_2, f_3)\right| \leq 
C2^{-m}\left\| \check{f_1}\right\|_1
    \left\|\Id_{m, \ell + \ell_0 } 
R_{\Phi_1} \check{f_2}  \right\|_2 
\left\| \Id_{m, \ell}\check{f_3}\right\|_2\,.
\end{equation}
In order to obtain a similar estimate for $\Lambda_{j,m, n,\ell_0, \ell,1}(f_1, f_2, f_3) $,
we change variables by $u=2^{-(d-1)j}x-2^m t $ and $v=x-2^m t^d$ to express 
$ \Lambda_{j,m, n,\ell_0, \ell,1}(f_1, f_2, f_3) $ as
$$
\int\!\!\!\int R_{\Phi_1} \check {f_1}(u)
\left(\Id_{m, \ell + \ell_0 } 
R_{\Phi_1} \check{f_2}\right) (v)  \rho (t(u,v)) 
 \left( \Id^*_{(d-1)j+m, n} \Id_{m, \ell} \check{f_3} \right)\!(x(u,v)) 
 \frac {du dv}{J(u,v)} \,,
$$
where  $J(u,v)$ is the Jocobian $ \frac{\partial (u, v)}{\partial(x,t)}$.
It is easy to see that the Jocobian $ \frac{\partial (u, v)}{\partial(x,t)}\sim 2^m$.
As we did for $ \Lambda_{j,m, n,\ell_0, \ell,1}$, we dominate the previous integral by 
$$
C2^{-m}
\int\!\left| R_{\Phi_1} \check {f_1}(u) \right|\left\|\Id_{m, \ell + \ell_0 } 
R_{\Phi_1} \check{f_2}  \right\|_2
 \!\left( \int \left| \left(\Id_{m, \ell} \check{f_3} \right)\!(x(u,v))
 \rho(t(u,v)) \right|^2      dv     \right)^{\frac{1}{2}} \!\!du .
$$
Notice that $|\partial x/\partial v| \sim 1$ whenever $t\in{\operatorname{supp}}\rho$. 
We then estimate 
 \begin{equation}\label{estLajmnkl1}
\left|\Lambda_{j,m, n,\ell_0, \ell,1}(f_1, f_2, f_3)\right| \leq 
C2^{-m}\left\| \check{f_1}\right\|_1
    \left\|\Id_{m, \ell + \ell_0 } 
R_{\Phi_1} \check{f_2}  \right\|_2 
\left\| \Id_{m, \ell}\check{f_3}\right\|_2\,,
\end{equation}
(\ref{estB1jmnl122}) follows from (\ref{estLajmnkl2}) and (\ref{estLajmnkl1}).
An interpolation of (\ref{estB1jmnl1}) and (\ref{estB1jmnl122}) then yields
\begin{equation}\label{estLajmn(1)}
  \left|\Lambda_{j,m, n}^{(1)}(f_1, f_2, f_3) \right|\leq 
C 2^{-\frac{(d-1)j}{2}-\frac{m}{2}} 2^{\frac{-(d-1)j+m}{2}}
\prod_{i=1}^3\|f_i\|_{L^2(\bI_i)}\,.
\end{equation}

We now turn to prove that
\begin{equation}\label{estLajmn(2)}
  \left|\Lambda_{j,m, n}^{(2)}(f_1, f_2, f_3) \right|\leq 
C_N 2^{-\frac{(d-1)j}{2}-\frac{m}{2}} \max\left\{2^{-100m}, \sigma \right\}
\prod_{i=1}^3\|f_i\|_{L^2(\bI_i)}\,.
\end{equation}
In dual frequency variables, $\Lambda_{j,m, n}^{(2)}(f_1, f_2, f_3) $
can be expressed as 
$$
\sum_{\ell_0=-10}^{10}\!\!\sum_{\ell}
2^{-\frac{(d-1)j}{2}}\!\!\!\iint f_1(\xi)\wh{\Phi_1}(\xi) e^{2\pi i 2^{-(d-1)j}\alpha_{m, \ell}\xi}
  \wh{F_{2,m, \ell_0, \ell}}(\eta)
 {\mathfrak{ m}}(\xi, \eta) \wh{F_{3,m,n, \ell}}(\eta)d\xi d\eta\,,
$$
where
\begin{equation}\label{defoffrakm}
 \mathfrak{m}(\xi, \eta) = \int \rho(t) e^{-2\pi i (2^m\xi t+ 2^m\eta t^d)}dt \,
\end{equation}
$$
 F_{2,m, \ell_0, \ell}=\Id_{m, \ell+\ell_0} R_{\Phi_1}\check{f_2} \,\,\, 
{\rm and}\,\,\,  F_{3,m, n, \ell} = \Id^*_{(d-1)j+m, n} \Id_{m, \ell} \check{f_3}\,.
$$
If $\eta$ is not in a small neighborhood of $\wh{\Phi_1}$, then there is 
no critical point  of the phase  function $\phi_{\xi, \eta}(t)=\xi t+ \eta t^d$
occurring in a small neighborhood of $\operatorname{supp}\rho $.
Integration by parts gives a rapid decay $ O(2^{-Nm})$ for $\mathfrak{m}  $.
Thus in this case, we dominate $\left|\Lambda_{j,m, n}^{(2)}(f_1, f_2, f_3) \right|  $ 
by
\begin{equation}\label{nocriest}
  C_N 2^{-N m} \prod_{i=1}^3\|f_i\|_{L^2(\bI_i)}\,,
\end{equation}
for any positive integer $N$.
We now only need to consider the worst case when there is a critical point 
of the phse function $\phi_{\xi, \eta}(t)=\xi t+ \eta t^d$ 
in a small neighborhood of $\operatorname{supp}\rho$. In this case, 
 $\eta$ must be in a small neighborhood of $\wh{\Phi_1}$ and 
the stationary phase
method gives
\begin{equation}\label{frakm111}
 \mathfrak{m}(\xi, \eta)\sim  2^{-m/2} e^{2\pi i c_d 2^m \eta^{-\frac{1}{d-1}} \xi^{d/(d-1)} } \,,
\end{equation}
where $c_d$ is a constant depending on $d$ only. Thus the principle term of 
$ \Lambda_{j,m, n}^{(2)}(f_1, f_2, f_3)$ is
$$
\sum_{\ell_0=-10}^{10}\!\!\!\sum_{\ell}
2^{-\frac{(d-1)j}{2}-\frac{m}{2}} \!
\iint f_1(\xi)\wh{\Phi_1}(\xi) e^{ i \phi_{d,m, \eta}(\xi)}
  \wh{F_{2,m, \ell_0, \ell}}(\eta)\wh{\Phi_2}(\eta)\wh{F_{3,m,n, \ell}}(\eta)d\xi d\eta
\,,
$$
where $\wh{\Phi_2}$ is a Schwartz function supported on a small neighborhood of 
$\wh{\Phi_1}$, and 
$$
\phi_{ d, m, \eta}(\xi)= 2\pi 
 c_d 2^{m} \eta^{-\frac{1}{d-1}} \xi^{d/(d-1)} + 2\pi 2^{-(d-1)j}\alpha_{m, \ell}\xi\,.
$$
The key point is that the integral in the previous expression can be viewed as 
an inner product of $F_{3,m,n, \ell}$ and $ \mathcal{M}F_{2, m, \ell_0, \ell}$,
where $\mathcal {M}$ is a multiplier operator defined by
$$
 \wh{\mathcal{M}f}(\eta) =  {\mathfrak m}_{d,j,m}(\eta) \wh{f}(\eta) \,. 
$$
Here the multiplier ${\mathfrak {m}}_{d,j,m}$ is given by
\begin{equation}\label{defofmultiplier}
{\mathfrak {m}}_{d,j,m}(\eta) = \int f_1(\xi)\wh{\Phi_1}(\xi)e^{i\phi_{d,m,\eta}(\xi)} d\xi\,.
\end{equation}
Observe that $\phi_{d,m,\eta}(\xi) +  b\xi $ is in $\mQ_1$ for any $b\in\mathbb R$ and 
$\eta\in \operatorname{supp}\wh{\Phi_2}$. Thus $\sigma$-uniformity in $\mQ_1$ of $f_1$ 
yields 
\begin{equation}\label{estofmultiplier}
\left\|{\mathfrak {m}}_{d,j,m} \right\|_\infty\leq C\sigma \left\|f_1\right\|_{L^2(\bI_1)}\,.
\end{equation}
And henceforth we dominate $\Lambda_{j,m, n}^{(2)}(f_1, f_2, f_3)$ by
$$
 \sum_{\ell_0=-10}^{10}\!\!\!\sum_{\ell}
2^{-\frac{(d-1)j}{2}-\frac{m}{2}} \sigma  \left\|f_1\right\|_{L^2(\bI_1)}
\left\|F_{2, m, \ell_0, \ell}\right\|_2
  \left\|F_{3, m, n, \ell}\right\|_2\,,
$$
which clearly is bounded by
\begin{equation}\label{final222La2jmn}
2^{-\frac{(d-1)j}{2}-\frac{m}{2}} \sigma \prod_{i=1}^3\left\|f_i\right\|_{L^2(\bI_i)}\,.
\end{equation}
Now (\ref{estLajmn(2)}) follows from 
(\ref{nocriest}) and (\ref{final222La2jmn}).
Combining (\ref{estLajmn(1)}) and  (\ref{estLajmn(2)}), we finish the proof.
\end{proof}

\begin{corollary}\label{cor1}
Let $ \Lambda_{j,m, n}(f_1, f_2, f_3)  $ be defined as in (\ref{defofLajmn00}).
Then 
 there exists a constant $C$ independent of $j, m, n$ such that
\begin{equation}\label{f1unifest221}
\left |  \Lambda_{j,m, n}(f_1, f_2, f_3) \right|\leq
 C \max\left\{ 2^{-100m}, 2^{\frac{-(d-1)j+m}{2}}, \sigma  \right\} 
\|f_1\|_{L^2(\bI_1)} \|f_2\|_{ L^2(\bI_1)} \|\wh{f_3}\|_\infty\,,
\end{equation} 
holds for all $f_1\in L^2(\bI_1)$ which are $\sigma$-uniform in $\mQ_1$,  
$f_2\in L^2(\bI_2)$ and $\wh {f_3} \in L^\infty$.
\end{corollary}

\begin{proof}
Since there is a smooth restriction factor $ \Id^*_{(d-1)j+m,n} $ in 
the definition of $\mathcal {B}_{j,m,m}$, the right hand side of (\ref{f1unifest}) 
can be sharpen to 
\begin{equation}\label{refinef1unifest}
 C 2^{-\frac{(d-1)j}{2}-\frac{m}{2}} \max\left\{2^{-100m}, 2^{\frac{-(d-1)j+m}{2}}, \sigma  \right\} 
\|f_1\|_{L^2(\bI_1)} \|f_2\|_{L^2(\bI_2)} \left\|\Id^{**}_{(d-1)j+m,n} \check{f_3}\right\|_{2}\,,
\end{equation}
which is clearly bounded by
$$
C \max\left\{ 2^{-100m}, 2^{\frac{-(d-1)j+m}{2}}, \sigma  \right\} 
\|f_1\|_{L^2(\bI_1)} \|f_2\|_{ L^2(\bI_1)} \|\wh{f_3}\|_\infty\,.
$$
\end{proof}

\begin{proposition}\label{prop2}
Let $ \Lambda_{j,m, n}(f_1, f_2, f_3)  $ be defined as in (\ref{defofLajmn00}).
Then there exists a constant $C$ independent of $j, m, n$ such that
\begin{equation}\label{q1unifest}
\left |  \Lambda_{j,m, n}(e^{iq_1}, f_2, f_3) \right|\leq
 C  2^{-\frac{\mathfrak{D}(d-1) m}{2}} \|f_2\|_{L^2(\bI_2)} \| \wh {f_3}\|_\infty \,,
\end{equation} 
holds for all $q_1\in\mQ_1$, $f_2\in L^2(\bI_2)$ and $\wh {f_3} \in L^\infty$,
where $\mathfrak{D}(d-1)$ is a positive constant defined in (\ref{defofmathfrakd1}).
\end{proposition}

A proof of Proposition \ref{prop2} will be provided in Section \ref{proofprop2}.

\section{Estimates of the trilinear forms, Case $j<0$}\label{triest000neg}
\setcounter{equation}0

Let $ \mQ_2 $ be a set of some functions defined by
\begin{equation}\label{defofmQ2neg}
 \mQ_2= \left\{ a \eta^{-\frac{1}{d-1}} + b \eta :  2^{m-100} \leq |a|\leq
 2^{m+100}  \,\,{\rm and}\,\, a, b\in\ZR
     \right\}\,.
\end{equation}

\begin{proposition}\label{prop1neg}
Let $f_2$ be  $\sigma$-uniform in $\mQ_2$. 
And let $j\leq 0$ and  $ \Lambda_{j,m, n}(f_1, f_2, f_3)  $ be defined as in (\ref{defofLajmn00}).
Then 
 there exists a constant $C$ independent of $j, m, n,  f_1$ such that
\begin{equation}\label{f1unifestneg}
\left |  \Lambda_{j,m, n}(f_1, f_2, f_3) \right|\leq
 C 2^{\frac{(d-1)j}{2}-\frac{m}{2}} \max\left\{2^{-100m},  2^{\frac{(d-1)j+m}{2}}, \sigma  \right\} 
\prod_{i=1}^3\|f_i\|_{L^2(\bI_i)}\,,
\end{equation} 
holds for all $f_1\in L^2(\bI_2)$ and $f_3\in L^2(\bI_3)$. 
\end{proposition}

\begin{proof}
Let $\Id_{m,l}=\Id_{\bI_{m,l}}$ and let ${\mathcal B}_{j,m,n, \ell}$ be a bilinear operator defined by
$$
 {\mathcal B}_{j,m,n, \ell}(f,g)(x) = {\mathcal B}_{j,m,n}(f,g)(x) \Id_{m,\ell}(x)\,,
$$
for all $f, g$.
Decompose $\Lambda_{j,m, n}(f_1, f_2, f_3)$ into
$\sum_{\ell} \Lambda_{j,m,n, \ell}$,
where 
$$ 
\Lambda_{j,m,n, \ell}(f_1, f_2, f_3) = 
\left\langle  {\mathcal B}_{j,m,n, \ell}(\check {f_1}, \check{f_2}), 
  \check {f_3} \right\rangle \,.
$$
Let $\alpha_{m, \ell}$ be a fixed point in the interval $\bI_{m, \ell}$.
And set $G_{\Phi_1,j,m,\ell}(x, t)$ to be 
$$
 G_{\Phi_1,j,m,\ell}(x, t):=  R_{\Phi_1} \check {f_2} (2^{(d-1)j} x - 2^m t^d)
     - R_{\Phi_1} \check{f_2}(2^{(d-1)j}\alpha_{m, \ell} - 2^m t^d) 
$$

Split ${\mathcal B}_{j,m,n, \ell}(\check{f_1}, \check{f_2})$ into two terms:
$$
 {\mathcal B}_{j,m,n, \ell}^{(1)}\left(\check{f_1}, \check{f_2}\right) + 
{\mathcal B}_{j,m,n, \ell}^{(2)} \left(\check{f_1}, \check{f_2}\right)\,,
$$
where ${\mathcal B}_{j,m,n, \ell}^{(1)}\left(\check{f_1}, \check{f_2}\right)$ is equal to
$$
2^{(d-1)j/2}\!\!
\int_{\mathbb R} R_{\Phi_1}\check{f_1} \left(x-2^m t  \right)
G_{\Phi_1,j,m,\ell}(x, t)  \rho(t) dt
 \left( \Id^*_{(d-1)|j|+m, n}(x)  \Id_{m, \ell}(x) \right)
$$
and ${\mathcal B}_{j,m,n, \ell}^{(2)} \left(\check{f_1}, \check{f_2}\right)$ equals to
$$
2^{-(d-1)j/2}\!\!
\int_{\mathbb R} \!R_{\Phi_1}\check{f_1} \left(x-2^m t  \right)
 R_{\Phi_1} {\check{f_2}}(2^{(d-1)j}\alpha_{m, \ell} - 2^m t^d) 
  \rho(t) dt
 \left( \Id^*_{(d-1)|j|+m, n}(x) \Id_{m, \ell}(x)\right)\,.
$$
For $i=1, 2$, let $ \Lambda_{j,m, n}^{(i)}(f_1, f_2, f_3)$ denote
$$
\sum_\ell  \left\langle {\mathcal B}_{j,m,n, \ell}^{(i)}\left(\check{f_1}, \check{f_2}\right), 
  \check{f_3} \right\rangle  \,.
$$
We now start to prove that 
\begin{equation}\label{estB1jmnl1neg}
\left|\Lambda_{j,m, n}^{(1)}(f_1, f_2, f_3) \right|
\leq  2^{\frac{(d-1)j}{2} } 2^{(d-1)j+m} \left\| \check{f_1}\right\|_2
           \left\|\check{f_2}\right\|_\infty \left\|\check{f_3}\right\|_2\,.
\end{equation} 
The mean value theorem and the smoothness of $\Phi_1$ yield 
that for $x\in\bI_{m,\ell}$, 
\begin{equation}\label{estdiffFjml1neg}
 \left| G_{\Phi_1,j,m,\ell}(x, t) \right|\leq C\left\| \check{f_2}\right\|_\infty
  2^{(d-1)j} \left| x-\alpha_{m, \ell}\right| \leq  C2^{(d-1)j+m}\left \| \check{f_2}\right\|_\infty \,.              
\end{equation} 
Because $|t|\sim 1$ when $t\in\operatorname{supp}\rho$, 
${\mathcal B}_{j,m,n, \ell}^{(1)}\left(\check{f_1}, \check{f_2}\right)$
can be written as
\begin{equation}\label{sumlk1111neg}
2^{\frac{(d-1)j}{2}}\!\!
\int_{\mathbb R} \!G_{\Phi_1,j,m,\ell}(x, t)
 \sum_{\ell_0}\left(\Id_{m, \ell + \ell_0 } 
R_{\Phi_1} \check{f_1}\right) (x - 2^{m} t ) \rho(t) dt
 \left( \Id^*_{(d-1)|j|+m, n}(x)  \Id_{m, \ell}(x) \right)\,,
\end{equation}
where $\ell_0$ is an integer between $-10$ and $10$.
Putting absolute value throughout and applying (\ref{estdiffFjml1neg})
plus 
 Cauchy-Schwarz inequality, we then estimate 
$ \left|\Lambda_{j,m, n}^{(1)}(f_1, f_2, f_3) \right|$ by
$$
 C2^{\frac{(d-1)j}{2}}2^{(d-1)j+m}\left\| \check{f_2}\right\|_\infty 
\sum_{\ell_0=-10}^{10} 
\sum_\ell\left\|\Id_{m, \ell + \ell_0 } 
R_{\Phi_1} \check{f_1}  \right\|_2 
\left\| \Id_{m, \ell}\check{f_3}\right\|_2\,,
$$
which clearly gives (\ref{estB1jmnl1neg}) by one more use of Cauchy-Schwarz inequality.\\

We now prove that 
\begin{equation}\label{estB1jmnl122neg}
\left|\Lambda_{j,m, n}^{(1)}(f_1, f_2, f_3) \right|
\leq  2^{\frac{(d-1)j}{2} } 2^{-m} \left\| \check{f_1}\right\|_2
           \left\|\check{f_2}\right\|_1 \left\|\check{f_3}\right\|_2\,.
\end{equation} 
From (\ref{sumlk1111neg}), we get that $\Lambda_{j,m, n}^{(1)}(f_1, f_2, f_3)$ equals to
$$
2^{\frac{(d-1)j}{2}}\sum_{\ell_0=-10}^{10}
\sum_\ell \Lambda_{j,m, n,\ell_0, \ell,1}(f_1, f_2, f_3) -  \Lambda_{j,m, n,\ell_0, \ell,2}(f_1, f_2, f_3)\,,
$$
where $\Lambda_{j,m, n,\ell_0, \ell,1}(f_1, f_2, f_3)$ is equal to
$$
\int_{{\mathbb R}^2} \!\! 
\!\left(\Id_{m, \ell + \ell_0 } 
R_{\Phi_1} \check{f_1}\right) \!(x \!- \!2^{m} t )
 R_{\Phi_1} \check {f_2} \left( 2^{(d-1)j} x \!- \!2^m t^d \right)\rho(t)\! 
 \left( \Id^*_{(d-1)|j|+m, n} \Id_{m, \ell} \check{f_3} \right)\!(x)  dt dx
$$
and $\Lambda_{j,m, n,\ell_0, \ell,2}(f_1, f_2, f_3)$  equals to
$$
\int_{{\mathbb R}^2} \!\! \!
\left(\Id_{m, \ell + \ell_0} 
R_{\Phi_1} \check{f_1}\right) \!\left(x - 2^{m} t\right) \!
R_{\Phi_1} \check {f_2} \left(2^{(d-1)j} \alpha_{m, \ell} - \!2^m t^d\right)\!\!
\rho(t)\! 
 \left( \Id^*_{(d-1)|j|+m, n} \Id_{m, \ell} \check{f_3} \right)\!(x) dt dx .
$$
 Cauchy-Schwarz inequality 
yields that 
\begin{equation}\label{estLajmnkl2neg}
\left|\Lambda_{j,m, n,\ell_0, \ell,2}(f_1, f_2, f_3)\right| \leq 
C2^{-m}
    \left\|\Id_{m, \ell + \ell_0 } 
R_{\Phi_1} \check{f_1}  \right\|_2  \left\| \check{f_2}\right\|_1
\left\| \Id_{m, \ell}\check{f_3}\right\|_2\,.
\end{equation}
In order to obtain a similar estimate for $\Lambda_{j,m, n,\ell_0, \ell,1}(f_1, f_2, f_3) $,
we change variables by $u= x-2^m t $ and $v=2^{(d-1)j} x-2^m t^d$ to express 
$ \Lambda_{j,m, n,\ell_0, \ell,1}(f_1, f_2, f_3) $ as
$$
\int\!\!\!\int 
\left(\Id_{m, \ell + \ell_0 } 
R_{\Phi_1} \check{f_1}\right) (u) R_{\Phi_1} \check {f_2}(v) \rho (t(u,v)) 
 \left( \Id^*_{(d-1)|j|+m, n} \Id_{m, \ell} \check{f_3} \right)\!(x(u,v)) 
 \frac {du dv}{J(u,v)} \,,
$$
where  $J(u,v)$ is the Jocobian $ \frac{\partial (u, v)}{\partial(x,t)}$.
It is easy to see that the Jocobian $ \frac{\partial (u, v)}{\partial(x,t)}\sim 2^m$.
As we did for $ \Lambda_{j,m, n,\ell_0, \ell,1}$, we dominate the previous integral by 
$$
C2^{-m}
\int\! \left\|\Id_{m, \ell + \ell_0 } 
R_{\Phi_1} \check{f_1}  \right\|_2 \left|R_{\Phi_1} \check {f_2}(v) \right|
 \!\left( \int \left| \left(\Id_{m, \ell} \check{f_3} \right)\!(x(u,v))
 \rho(t(u,v)) \right|^2      du     \right)^{\frac{1}{2}} \!\!dv .
$$
Notice that $|\partial x/\partial u| \sim 1$ whenever $t\in{\operatorname{supp}}\rho$. 
We then estimate 
 \begin{equation}\label{estLajmnkl1neg}
\left|\Lambda_{j,m, n,\ell_0, \ell,1}(f_1, f_2, f_3)\right| \leq 
C2^{-m}
    \left\|\Id_{m, \ell + \ell_0 } 
R_{\Phi_1} \check{f_1}  \right\|_2 \left\| \check{f_2}\right\|_1
\left\| \Id_{m, \ell}\check{f_3}\right\|_2\,,
\end{equation}
(\ref{estB1jmnl122neg}) follows from (\ref{estLajmnkl2neg}) and (\ref{estLajmnkl1neg}).
An interpolation of (\ref{estB1jmnl1neg}) and (\ref{estB1jmnl122neg}) then yields
\begin{equation}\label{estLajmn(1)neg}
  \left|\Lambda_{j,m, n}^{(1)}(f_1, f_2, f_3) \right|\leq 
C 2^{\frac{(d-1)j}{2}-\frac{m}{2}} 2^{\frac{(d-1)j+m}{2}}
\prod_{i=1}^3\|f_i\|_{L^2(\bI_i)}\,.
\end{equation}

We now turn to prove that if $f_2$ is $\sigma$-uniform in $\mQ_2$, then  
\begin{equation}\label{estLajmn(2)neg}
  \left|\Lambda_{j,m, n}^{(2)}(f_1, f_2, f_3) \right|\leq 
C_N 2^{\frac{(d-1)j}{2}-\frac{m}{2}} \max\left\{2^{-100m}, \sigma \right\}
\prod_{i=1}^3\|f_i\|_{L^2(\bI_i)}\,.
\end{equation}
In dual frequency variables, $\Lambda_{j,m, n}^{(2)}(f_1, f_2, f_3) $
can be expressed as 
$$
\sum_{\ell_0=-10}^{10}\!\!\sum_{\ell}
2^{\frac{(d-1)j}{2}}\!\!\!\iint 
 \wh{F_{1,m, \ell_0, \ell}}(\xi)
f_2(\eta)\wh{\Phi_1}(\eta) e^{2\pi i 2^{(d-1)j}\alpha_{m, \ell}\eta}
  {\mathfrak{ m}}(\xi, \eta) \wh{F_{3,m,n, \ell}}(\xi)d\xi d\eta\,,
$$
where
\begin{equation}\label{defoffrakmneg}
 \mathfrak{m}(\xi, \eta) = \int \rho(t) e^{-2\pi i (2^m\xi t+ 2^m\eta t^d)}dt \,
\end{equation}
$$
 F_{1,m, \ell_0, \ell}=\Id_{m, \ell+\ell_0} R_{\Phi_1}\check{f_1} \,\,\, 
{\rm and}\,\,\,  F_{3,m, n, \ell} = \Id^*_{(d-1)|j|+m, n} \Id_{m, \ell} \check{f_3}\,.
$$
If $\xi$ is not in a small neighborhood of $\wh{\Phi_1}$, then there is 
no critical point  of the phse function $\phi_{\xi, \eta}(t)=\xi t+ \eta t^d$
occurring in a small neighborhood of $\operatorname{supp}\rho $.
Integration by parts gives a rapid decay $ O(2^{-Nm})$ for $\mathfrak{m}  $.
Thus in this case, we dominate $\left|\Lambda_{j,m, n}^{(2)}(f_1, f_2, f_3) \right|  $ 
by
\begin{equation}\label{nocriestneg}
  C_N 2^{-N m} \prod_{i=1}^3\|f_i\|_{L^2(\bI_i)}\,,
\end{equation}
for any positive integer $N$.
We now only need to consider the worst case when there is a critical point 
of the phse function $\phi_{\xi, \eta}(t)=\xi t+ \eta t^d$ 
in a small neighborhood of $\operatorname{supp}\rho$. In this case, 
 $\xi$ must be in a small neighborhood of $\wh{\Phi_1}$ and 
the stationary phase
method gives
\begin{equation}\label{frakm111neg}
 \mathfrak{m}(\xi, \eta)\sim  2^{-m/2} e^{2\pi i c_d 2^m \xi^{d/(d-1)}\eta^{-\frac{1}{d-1}}  } \,,
\end{equation}
where $c_d$ is a constant depending on $d$ only. Thus the principle term of 
$ \Lambda_{j,m, n}^{(2)}(f_1, f_2, f_3)$ is
$$
\sum_{\ell_0=-10}^{10}\!\!\!\sum_{\ell}
2^{\frac{(d-1)j}{2}-\frac{m}{2}} \!
\iint \wh{F_{1,m, \ell_0, \ell}}(\xi)\wh{\Phi_2}(\xi)
f_2(\eta)\wh{\Phi_1}(\eta) e^{ i \phi_{d,m, \xi}(\eta)}
  \wh{F_{3,m,n, \ell}}(\xi)d\xi d\eta
\,,
$$
where $\wh{\Phi_2}$ is a Schwartz function supported on a small neighborhood of 
$\wh{\Phi_1}$, and 
$$
\phi_{ d, m, \xi}(\eta)= 2\pi 
 c_d 2^{m} \xi^{d/(d-1)} \eta^{-\frac{1}{d-1}}  + 2\pi 2^{(d-1)j}\alpha_{m, \ell}\eta\,.
$$
The key point is that the integral in the previous expression can be viewed as 
an inner product of $F_{3,m,n, \ell}$ and $ \mathcal{M}F_{1, m, \ell_0, \ell}$,
where $\mathcal {M}$ is a multiplier operator defined by
$$
 \wh{\mathcal{M}f}(\xi) =  {\mathfrak m}_{d,j,m}(\xi) \wh{f}(\xi) \,. 
$$
Here the multiplier ${\mathfrak {m}}_{d,j,m}$ is given by
\begin{equation}\label{defofmultiplierneg}
{\mathfrak {m}}_{d,j,m}(\xi) = \int f_2(\eta)\wh{\Phi_1}(\eta)e^{i\phi_{d,m,\xi}(\eta)} d\eta\,.
\end{equation}
Observe that $\phi_{d,m,\xi}(\eta) +  b\eta$ is in $\mQ_2$ for any $b\in\mathbb R$ and 
$\xi\in \operatorname{supp}\wh{\Phi_2}$. Thus $\sigma$-uniformity in $\mQ_2$ of $f_2$ 
yields 
\begin{equation}\label{estofmultiplierneg}
\left\|{\mathfrak {m}}_{d,j,m} \right\|_\infty\leq C\sigma \left\|f_2\right\|_{L^2(\bI_2)}\,.
\end{equation}
And henceforth we dominate $\Lambda_{j,m, n}^{(2)}(f_1, f_2, f_3)$ by
$$
 \sum_{\ell_0=-10}^{10}\!\!\!\sum_{\ell}
2^{\frac{(d-1)j}{2}-\frac{m}{2}} \sigma  \left\|f_2\right\|_{L^2(\bI_2)}
\left\|F_{1, m, \ell_0, \ell}\right\|_2
  \left\|F_{3, m, n, \ell}\right\|_2\,,
$$
which clearly is bounded by
\begin{equation}\label{final222La2jmnneg}
2^{\frac{(d-1)j}{2}-\frac{m}{2}} \sigma \prod_{i=1}^3\left\|f_i\right\|_{L^2(\bI_i)}\,.
\end{equation}
Now (\ref{estLajmn(2)neg}) follows from 
(\ref{nocriestneg}) and (\ref{final222La2jmnneg}).
Combining (\ref{estLajmn(1)neg}) and  (\ref{estLajmn(2)neg}), we finish the proof.
\end{proof}

\begin{corollary}\label{cor1neg}
Let $j\leq 0$ and $ \Lambda_{j,m, n}(f_1, f_2, f_3)  $ be defined as in (\ref{defofLajmn00}).
Then 
 there exists a constant $C$ independent of $j, m, n$ such that
\begin{equation}\label{f1unifest221neg}
\left |  \Lambda_{j,m, n}(f_1, f_2, f_3) \right|\leq
 C \max\left\{ 2^{-100m}, 2^{\frac{(d-1)j+m}{2}}, \sigma  \right\} 
\|f_1\|_{L^2(\bI_1)} \|f_2\|_{ L^2(\bI_1)} \|\wh{f_3}\|_\infty\,,
\end{equation} 
holds for all $f_2\in L^2(\bI_2)$ which are $\sigma$-uniform in $\mQ_2$,  
$f_1\in L^2(\bI_1)$ and $\wh {f_3} \in L^\infty$.
\end{corollary}

\begin{proof}
Since there is a smooth restriction factor $ \Id^*_{(d-1)|j|+m,n} $ in 
the definition of $\mathcal {B}_{j,m,m}$, the right hand side of (\ref{f1unifestneg}) 
can be sharpen to 
\begin{equation}\label{refinef1unifestneg}
 C 2^{\frac{(d-1)j}{2}-\frac{m}{2}} \max\left\{2^{-100m}, 2^{\frac{(d-1)j+m}{2}}, \sigma  \right\} 
\|f_1\|_{L^2(\bI_1)} \|f_2\|_{L^2(\bI_2)} \left\|\Id^{**}_{(d-1)|j|+m,n} \check{f_3}\right\|_{2}\,,
\end{equation}
which is clearly bounded by
$$
C \max\left\{ 2^{-100m}, 2^{\frac{(d-1)j+m}{2}}, \sigma  \right\} 
\|f_1\|_{L^2(\bI_1)} \|f_2\|_{ L^2(\bI_2)} \|\wh{f_3}\|_\infty\,.
$$
\end{proof}

\begin{proposition}\label{prop2neg}
Let $j\leq 0 $ and $ \Lambda_{j,m, n}(f_1, f_2, f_3)  $ be defined as in (\ref{defofLajmn00}).
Then there exists a constant $C$ independent of $j, m, n$ such that
\begin{equation}\label{q1unifestneg}
\left |  \Lambda_{j,m, n}(f_1, e^{iq_2}, f_3) \right|\leq
 C  2^{-m/4} \|f_1\|_{L^2(\bI_1)} \| \wh {f_3}\|_\infty \,,
\end{equation} 
holds for all $q_2\in\mQ_2$, $f_1\in L^2(\bI_1)$ and $\wh {f_3} \in L^\infty$.
\end{proposition}

We will prove Proposition \ref{prop2neg} in Section \ref{proofprop2neg}. We now are ready to 
provide a proof of Theorem {\ref{thmtriest1}}. \\

\subsection{Proof of Theorem {\ref{thmtriest1}}}

Corollary {\ref{cor1}}, Proposition \ref{prop2} and Theorem {\ref{uniformitythm}} yield
that $ \left |  \Lambda_{j,m, n}(f_1, f_2, f_3) \right|$ is dominated by 
 \begin{equation}\label{f1unifest221ff}
 C\left( \max\left\{ 2^{-100m}, 2^{\frac{-(d-1)j+m}{2}}, \sigma  \right\} +  \frac{2^{-\mathfrak{D}(d-1)m/2 } }{\sigma}   
 \right)
\|f_1\|_{L^2(\bI_1)} \|f_2\|_{ L^2(\bI_1)} \|\wh{f_3}\|_\infty\,,
\end{equation} 
holds for all $f_1\in L^2(\bI_1)$,  
$f_2\in L^2(\bI_2)$ and $\wh {f_3} \in L^\infty$. 
Take $\sigma$ to be $2^{ -\mathfrak{D}(d-1)m/4}$, then we have 
\begin{equation}\label{f1unifest221fff}
\left |  \Lambda_{j,m, n}(f_1, f_2, f_3) \right|\leq
 C \max\left\{ 2^{\frac{-(d-1)j+m}{2}}, 2^{ -\mathfrak{D}(d-1)m/4} \right\} 
\|f_1\|_{L^2(\bI_1)} \|f_2\|_{ L^2(\bI_1)} \|\wh{f_3}\|_\infty\,.
\end{equation}  
This give a proof for the case $j>0$. For the case $j\leq 0$, 
applying Corollary {\ref{cor1neg}}, Proposition \ref{prop2neg} and Theorem {\ref{uniformitythm}},
we estimate $ \left |  \Lambda_{j,m, n}(f_1, f_2, f_3) \right| $ by
\begin{equation}\label{f1unifest221ffneg}
 C\left( \max\left\{ 2^{-100m}, 2^{\frac{(d-1)j+m}{2}}, \sigma  \right\} +  \frac{2^{-m/4 } }{\sigma}   
 \right)
\|f_1\|_{L^2(\bI_1)} \|f_2\|_{ L^2(\bI_1)} \|\wh{f_3}\|_\infty\,,
\end{equation} 
holds for all $f_1\in L^2(\bI_1)$,  
$f_2\in L^2(\bI_2)$ and $\wh {f_3} \in L^\infty$. 
Now choose $\sigma$ to be $2^{ -m/8}$. Then we have 
\begin{equation}\label{f1unifest221ffnegf}
\left |  \Lambda_{j,m, n}(f_1, f_2, f_3) \right|\leq
 C \max\left\{ 2^{\frac{(d-1)j+m}{2}}, 2^{ -m/8} \right\} 
\|f_1\|_{L^2(\bI_1)} \|f_2\|_{ L^2(\bI_1)} \|\wh{f_3}\|_\infty\,,
\end{equation}  
which completes the proof of the case $j\leq 0$. Therefore 
 combining (\ref{f1unifest221fff}) and (\ref{f1unifest221ffnegf}), we proved Theorem {\ref{thmtriest1}}.

\section{Proof of Proposition \ref{prop2}}\label{proofprop2}
\setcounter{equation}0

\begin{lemma}\label{vdcorput2d}
Let $\ell\geq 1$. Let $\bI_1$ and $\bI_2$ be fixed bounded intervals. And let 
$\varphi$  be a function from $\bI_1\times \bI_2$ to $\mathbb R$ satisfying 
\begin{equation}\label{mixderivative}
 \left| \partial^{\ell}_x\partial_y \varphi(x,y)\right|\geq 1\,,
\end{equation}  
for all $(x,y)\in \bI_1\times \bI_2$. 
Assume an additional condition holds in the case $\ell=1$,
\begin{equation}\label{mixderivative1}
 \left| \partial^{2}_x\partial_y \varphi(x,y)\right|\neq 0\,,
\end{equation} 
for all $(x,y)\in \bI_1\times \bI_2$.
Then there exists a constant depending on the length of $\bI_1$ and $\bI_2$ but
independent of $\varphi, \lambda$ and  the locations of $\bI_1$ and $ \bI_2$  
such that
\begin{equation}\label{IestCorput}
 \left|\iint_{\bI_1\times \bI_2} e^{i\lambda \varphi(x,y)} f(x) g(x) 
dx dy \right|\leq C 
(1+|\lambda|)^{-\mathfrak{D}(\ell)} \|f\|_2\|g\|_2 \,,
\end{equation}
for all $f, g\in L^2$, where 
\begin{equation}\label{defofmathfrakd1}
\mathfrak{D}(\ell) = \left\{
\begin{array}{ll}
 {1}/({2\ell}), & {\rm if} \,\,  \ell\geq 2\,; \\
 1/({2 + \e}), & {\rm if} \,\,\ell=1\,  \,.
\end{array}
\right. 
\end{equation}
for any $\e>0$.
\end{lemma}

This lemma is related to a 2-dimensional van der Corput lemma proved in \cite{CCW}.
The case $\ell\geq 2$ was proved in \cite{CCW}. And a proof of the case $\ell=1$ can be found
in \cite{PSt1}. The estimates on $\mathfrak{D}(\ell)$ in (\ref{defofmathfrakd1}) are not sharp.
With some additional convexity conditions on the phase function $\varphi$, $\mathfrak{D}(\ell) $ 
might be improved to be $1/(\ell+1)$ (see \cite{CCW} for some of such improvements). But in 
this article we do not need to pursue the sharp estimates.

\begin{lemma}\label{mixdevarphi}
Let $c, \tau\in \mathbb R$ and $\varphi$ be a function defined by
\begin{equation}\label{defofvarphic}
\varphi_c(x,y) = \left(x-y^{1/d} + c\right)^d
\end{equation}
Define $\bQ_{c,j.\tau}(x,y)$ by
\begin{equation}\label{defofmQcjtau}
 \bQ_{c,j, \tau}(x,y)=\varphi_c(x, y) - \varphi_c(x+2^{-(d-1)j}\tau, y+\tau)\,.
\end{equation}
Then there exists a constant $C_d$ depending only on $d$ such that 
\begin{equation}\label{estmixd11}
 \left| \partial_x^{d-1}\partial_y \bQ_{c,j, \tau}(x,y)\right|\geq C_d|\tau|\,
\end{equation}
holds for all $y, y+\tau\in [2^{-100}, 2^{100}]$. Moreover, if $d=2$,
we have 
\begin{equation}\label{estmixd12}
 \left| \partial_x\partial^2_y \bQ_{c,j, \tau}(x,y)\right|\geq  C_d |\tau|\,
\end{equation}
holds for all $y, y+\tau\in [2^{-100}, 2^{100}]$.
\end{lemma}

\begin{proof}
A direct computation yields 
\begin{equation}
 \partial_x^{d-1}\partial_y \bQ_{c,j, \tau}(x,y) = C_d \left( (y+\tau)^{\frac{1}{d}-1} - 
   y^{\frac{1}{d}-1} \right)\,.
\end{equation}
Hence the desired estimate (\ref{estmixd11}) follows immediately from the mean value theorem. 
(\ref{estmixd12}) can be obtained similarly. 
\end{proof}

\begin{lemma}\label{vdcorput2d111}
 Let $\bI$ be  a fixed interval of length $1$.
And let $\theta$ be a bump function supported on $[1/100, 2]$ (or $[-2, -1/100]$). 
Suppose that $\phi_{d, j,m}$ is a phase function defined by
\begin{equation}\label{defofphidjm}
 \phi_{d,j,m}(x,y) = C_{d,j,m}
 2^{m} \left(x-  y^{\frac{1}{d}} + c_{j,m}\right)^{d}\,,
\end{equation}
where $C_{d,j,m}, c_{j,m}$ are constants  independent of 
$x, y$ such that  $ 2^{-200}\leq |C_{d,j,m}|\leq 2^{200}$. 
Let $\Lambda_{d,j,m, \bI}$ be a bilinear form defined by
\begin{equation}
\Lambda_{d,j,m, \bI}(f,g)= 
\iint e^{i \phi_{d,j,m}(x,t) } f\left(x-2^{-(d-1)j}t\right)  g(x)
 \Id_\bI(x) \theta(t) dx dt\,.
 \end{equation}
Then we have 
\begin{equation}\label{estCorputphidjm}
 \left|\Lambda_{d,j,m, \bI}(f,g)  \right|  \leq C_d 2^{-\frac{\mathfrak{D}(d-1) m}{2}} \|f\|_2\|g\|_\infty \,,
\end{equation}
holds for all $f\in L^2$ and $g\in L^\infty$, 
where $C_d$ is a constant depending only on $d$.
\end{lemma}

\begin{proof}

The bilinear form $\Lambda_{d,j,m, \bI}(f,g)$ equals to 
$
 \left\langle  {\bT}_{d,j,m, \bI}(g), f\right\rangle
$,
where $\bT_{d,j,m, \bI}$ is defined by
\begin{equation}
\bT_{d,j,m, \bI}g(x) = \int e^{i \phi_{d,j,m}(x + 2^{-(d-1)j}t,t) }
\left(g\Id_\bI\right) \left(x+2^{-(d-1)j}t\right) \theta(t) dt\,.
\end{equation}
By a change of variables, $ \|\bT_{d,j,m, \bI}g\|^2_2$ can be expressed as
$$
 \int\left( \iint e^{i \Phi_{d,j,m, \tau}(x,t)}
 G_{\tau} \left(x+2^{-(d-1)j}t\right) 
 \Theta_\tau(t)
             dx dt \right) d\tau\,,
$$
where
$$
 \Phi_{d,j,m, \tau}(x,t)=\phi_{d,j,m}(x+2^{-(d-1)j}t,t)- 
 \phi_{d,j,m}(x+2^{-(d-1)j}t+2^{-(d-1)j}\tau , t+\tau)
$$
$$
G_{\tau}(x)= 
\left(\Id_{\bI}g\right)(x) \overline{ \left(\Id_{\bI}g\right) \left(x+2^{-(d-1)j}\tau)\right) }\,,
$$
$$
\Theta_\tau(t) = \theta(t)\theta(t+\tau)\,.
$$
Changing coordinates $(x,t)\mapsto (u,v)$ by 
 $u= x+2^{-(d-1)j}t$ and $v=t$, we write the inner double-integral 
in the previous integral as
$$
\iint e^{i C_{d,j,m}2^m \bQ_{c_{j,m},j, \tau}(u,v)} G_\tau(u) \Theta_\tau(v) du dv\,,
$$
where $\bQ_{c_{j,m}, j, \tau}$ is defined as in (\ref{defofmQcjtau}).
From (\ref{estmixd11}), (\ref{estmixd12}) and Lemma \ref{vdcorput2d},
we then estimate $ \|\bT_{d,j,m, \bI}g\|^2_2 $ by
$$
C_d\int_{-10}^{10} \min\left\{1,  2^{-\mathfrak{D}(d-1) m }
  \tau^{-\mathfrak{D}(d-1)}\right\}  \left\| G_\tau\right\|_2 \left\| \Theta_\tau\right\|_2 d\tau\,,
$$
which clearly is bounded by
$$
 C_d 2^{-\mathfrak{D}(d-1) m } \|g\|_\infty^2\,,
$$
Hence (\ref{estCorputphidjm}) follows and 
therefore we complete the proof.
\end{proof}

We now turn to the proof of Proposition {\ref{prop2}}.
For simplicity, we assume $\rho$ is supported on $[1/8, 2]$.
 For any function $q_1
 = (a\xi^{\frac{d}{d-1}} + b\xi \in\mQ_1$, 
we have 
\begin{equation}\label{q1inverse}
R_{\Phi_1}\Check{(e^{iq_1})}(x) = \int \wh{\Phi_1}(\xi) e^{ia\xi^{\frac{d}{d-1}}}e^{i(x+b)\xi} d\xi\,,
\end{equation}
where $|a|\sim 2^m$.  The stationary phase method yields that the principle part of
(\ref{q1inverse}) is 
\begin{equation}\label{prinpart}
 \mathcal{P}(q_1)(x)=  C_d |a|^{-1/2} e^{i c_1 a^{-(d-1)}(x+ b)^d } \wh{\Phi_1}\left(c_2 a^{-(d-1)}(x+b)^{d-1}\right)\,,
\end{equation}
where $C_d, c_1, c_2$ are constants depending only on $d$. 
Thus to obtain Proposition {\ref{prop2}}, it suffices to prove that there 
exists a constant $C$ such that
\begin{equation}\label{q1unif23estti}
\left |  \ti{\Lambda}_{j,m, n}(e^{iq_1}, f_2, f_3) \right|\leq
 C  2^{-\frac{\mathfrak{D}(d-1)m}{2}}  \|\check{f_2}\|_2 \| \check{f_3}\|_\infty \,,
\end{equation} 
holds for  all $q_1\in\mQ_1, \check{f_2}\in L^2$, and  $\check{f_3}\in L^\infty$, 
where $ \ti{\Lambda}_{j,m, n}(e^{iq_1}, f_2, f_3)$ is defined to be 
$$
 2^{-\frac{(d-1)j}{2}} \iint \mathcal{P}(q_1)\left(2^{-(d-1)j}x-2^m t\right) \check{f_2}\left(x-2^m t^d\right)
 \left(\Id^*_{(d-1)j+m,n}\check{f_3}\right)(x) \rho(t) dt dx\,. 
 $$
Observe that $\wh\Phi_1$ is supported essentially in a bounded interval away from $0$.
Thus we can restrict the variable $x$ in a bounded interval $\bI_{d,j,m}$ whose length is comparable 
to $2^{(d-1)j+m}$ and reduce the problem to estimate
\begin{equation}\label{q1unif23estforsub}
\left | {\Lambda}_{j,m, n, \bI_{d,j,m}}(f_2, f_3) \right|\leq
 C  2^{-\frac{\mathfrak{D}(d-1) m}{2}}  \|\check{f_2}\|_2 \| \check{f_3}\|_\infty \,,
\end{equation} 
holds for an absolute constant $C$ and all $\check{f_2}\in L^2, \check{f_3}\in L^\infty$,
where ${\Lambda}_{j,m, n, \bI_{d,j,m}}(f_2, f_3)$ is equal to
\begin{equation}
 2^{-\frac{(d-1)j}{2}-\frac{m}{2}}\!\iint \! \mathcal{P}_{d,j,m}\left( 2^{-(d-1)j}x-2^m t\right) \check{f_2}\left(x-2^m t^d\right)
 \left(\Id_{\bI_{d,j,m}}\check{f_3} \right)(x) \rho(t) dt dx\,.
\end{equation}
Here 
\begin{equation}
\mathcal{P}_{d,j,m}(x)= e^{i c_1a^{-(d-1)}(x+b)^d} \wh{\Phi_1} \left(c_2 a^{-(d-1)}(x+b)^{d-1}\right)\,.
\end{equation}
Let $\bI$ be an interval of length $1$.  A rescaling argument then reduces (\ref{q1unif23estforsub})
to an estimate of a bilinear form ${\Lambda}_{j,m, n, \bI}$ associated to $\bI$, that is, 
\begin{equation}\label{q1unif23estforunit}
\left | {\Lambda}_{j,m, n, \bI}(f, g) \right|\leq
 C 2^{-\frac{\mathfrak{D} (d-1)m}{2}}  \|f\|_2\| g\|_\infty \,,
\end{equation}
where ${\Lambda}_{j,m, n, \bI}(f, g)$ is defined by
$$
 \iint  \mathcal{P}_{d,j,m}\left( 2^m x-2^m t\right)
 f\left(x-2^{-(d-1)j} t^d\right) g(x) \Id_{\bI}(x)\rho(t) dt dx\,.
$$
Notice that 
\begin{equation}
\mathcal{P}_{d,j,m}\left( 2^m x-2^m t\right) = e^{i C_{d,j, m}2^m (x- t + c_{j,m})^d }\wh{\Phi_1} 
\left( C_d C_{d,m} (x -t + c_{m} )^{d-1}\right)\,,
\end{equation}
where $C_{d,j,m}, C_{d,m},  c_{j,m}, c_m, C_d$ are constants such that 
$|C_{d,j,m}|, |C_{d,m}|\in [2^{-100}, 2^{100}]$.
$\wh{\Phi_1} 
\left( C_d C_{d,m} (x -t + c_{m} )^{d-1}\right)$ can be dropped by utilizing Fourier series
since $\wh\Phi_1$ is a Schwartz function, 
because $x\in\bI, t\in{\operatorname{supp}\rho}$ are restricted in bounded intervals respectively. 
Then (\ref{q1unif23estforunit}) can be reduced to Lemma {\ref{vdcorput2d111}} by 
a change of variable $t^d\mapsto t$. Therefore we prove Proposition {\ref{prop2}}. \\

\section{Proof of Proposition \ref{prop2neg}}\label{proofprop2neg}
\setcounter{equation}0

\begin{lemma}\label{phasetauneg}
Let $j\leq 0$ and $\tau\in [-100, 100]$. And let $\phi_{d,j,m, \tau}$ be defined by
\begin{equation}\label{defofphidjmtau}  
 \phi_{d,j,m, \tau}(u,v) =  \left( u-v^d\right)^{1/d} -
    \left( u +2^{(d-1)j}\tau - (v+\tau)^d\right)^{1/d}\,.
\end{equation}
Suppose that $|j|\geq m/(d-1)$ and $|u|\geq 2^{-m}$. Then 
\begin{equation}\label{largederiofphi}
 \left| \partial_v\phi_{d,j,m, \tau}(u,v) \right|\geq C|\tau u|\, 
\end{equation}
holds whenever $v, v+\tau\in \bI_1$
and $u-v^d, u-(v+\tau)^d\in \bI_2$, where $\bI_i$ is the interval 
$[1/100, 100]$ or $[-100, -1/100]$.  
\end{lemma}

\begin{proof}
Clearly
$$
 \partial_v\phi_{d,j,m, \tau}(u,v)=
 - \left( u-v^d\right)^{\frac{1}{d}-1} v^{d-1} + 
 \left( u +2^{(d-1)j}\tau - (v+\tau)^d\right)^{\frac{1}{d}-1} (v+\tau)^{d-1}\,,
$$
which can be written as a sum    
$$
 \Phi_{d,j,m, \tau,1}(u,v) + \Phi_{d,j,m,\tau,2}(u,v)\,,
$$
where $\Phi_{d,j,m, \tau,1}(u,v)$ is 
$$
  - \left( u-v^d\right)^{\frac{1}{d}-1} v^{d-1} + 
 \left( u - (v+\tau)^d\right)^{\frac{1}{d}-1} (v+\tau)^{d-1}\,
$$
and $\Phi_{d,j,m, \tau,2}(u,v)  $ is equal to
$$
-\left( u + 2^{(d-1)j}\tau   - (v+\tau)^d\right)^{\frac{1}{d}-1} (v+\tau)^{d-1}  + 
 \left( u - (v+\tau)^d\right)^{\frac{1}{d}-1} (v+\tau)^{d-1}\,.
$$
The mean value theorem yields that 
\begin{equation}\label{smallphi2tau}
 \left| \Phi_{d,j,m, \tau,2}(u,v) \right|\leq C2^{(d-1)j}|\tau|\,,
\end{equation} 
and 
$$
 \Phi_{d,j,m, \tau,1}(u,v) = G'_{d,j,m, \tau}(\eta) \tau\,,
$$
where 
\begin{equation}\label{defofGdjmtau}
 G_{d,j,m,\tau}(v) = - \left( u-v^d\right)^{\frac{1}{d}-1} v^{d-1}\,
\end{equation}
and $\eta $ is a point between $v$ and $v+\tau$. A simple computation gives 
\begin{equation}\label{deriofGtau}
 G_{d,j,m,\tau}'(v)= -(d-1)\left(u-v^d\right)^{\frac{1}{d}-2}v^{d-2}u \,.
\end{equation}
Now (\ref{largederiofphi}) follows from (\ref{smallphi2tau}), (\ref{deriofGtau}), $|j|\geq m/(d-1)$ 
and $|u|\geq 2^{-m}$. Therefore we finish the proof. 
\end{proof}

\begin{lemma}\label{vdcorputneg}
Let $\theta_1, \theta_2$ be  bump functions supported on $\bI_1$ and
$\bI_2$ respectively, where $\bI_i$ is $[1/50, 2]$ or $[-2, -1/50]$. 
Suppose that $ j\leq 0$, $|j|\geq m/(d-1)$ and $\phi_{d, j,m}$ is a phase function defined by
\begin{equation}\label{defofphidjmneg}
 \phi_{d,j,m}(x, t) = C_{d,j,m}
 2^{m} \left(x- t^d\right)^{1/d}\,,
\end{equation}
where $C_{d,j,m}$ are constants  independent of 
$x, t$ such that  $ 2^{-200}\leq |C_{d,j,m}|\leq 2^{200}$. 
Let $\Lambda_{d,j,m}$ be a bilinear form defined by
\begin{equation}
\Lambda_{d,j,m}(f,g)= 
\iint e^{i \phi_{d,j,m}(x,t) } f\left(x-2^{(d-1)j}t\right)  g(x)
 \theta_1(x-t^d) \theta_2(t) dx dt\,.
 \end{equation}
Then we have 
\begin{equation}\label{estCorputneg}
 \left|\Lambda_{d,j,m}(f,g)  \right|  \leq C_d 2^{-m/4} 
\|f\|_2\|g\|_\infty \,,
\end{equation}
holds for all $f\in L^2$ and $g\in L^\infty$, 
where $C_d$ is a constant depending only on $d$.
\end{lemma}

\begin{proof}
The bilinear form $\Lambda_{d,j,m}(f,g)$ equals to 
$
 \left\langle  {\bT}_{d,j,m}(g), f\right\rangle
$,
where $\bT_{d,j,m}$ is defined by
\begin{equation}
\bT_{d,j,m}g(x) = \int e^{i \phi_{d,j,m}(x + 2^{(d-1)j}t,t) }
g\left(x+2^{(d-1)j}t\right) \theta_1\left(x+ 2^{(d-1)j}t -t^d\right) \theta_2(t) dt\,.
\end{equation}
By a change of variables, we express  $ \|\bT_{d,j,m}g\|^2_2$ as
$$
 \int\left( \iint e^{i \Phi_{d,j,m, \tau}(x + 2^{(d-1)j}t,t)}
 G_{\tau} \left(x+2^{(d-1)j}t\right) \Theta_{1,\tau}\left(x+2^{(d-1)j}t, t\right) 
 \Theta_{2,\tau}(t)   
dx dt \right)    d\tau\,,
$$
where
$$
 \Phi_{d,j,m, \tau}(x,t)=\phi_{d,j,m}(x,t)- 
 \phi_{d,j,m}(x+2^{(d-1)j}\tau,  t+\tau)
$$
$$
G_{\tau}(x)= g(x) \overline{ g\left(x+2^{(d-1)j}\tau)\right) }\,,
$$
$$
 \Theta_{1,\tau}\left(x, t\right)=
\theta_1\left(x -t^d\right)
 \overline{\theta_1\left(x + 2^{(d-1)j}\tau -(t+\tau)^d\right)  }\,,
$$
$$
\Theta_{2,\tau}(t) = \theta_2(t)\theta_2(t+\tau)\,.
$$
Changing coordinates $(x,t)\mapsto (u,v)$ by 
 $u= x+2^{(d-1)j}t$ and $v=t$, we write the inner double-integral 
in the previous integral as
\begin{equation}\label{*}
\iint e^{i \Phi_{d,j,m, \tau}(u,v)} G_\tau(u)\Theta_{1, \tau}(u,v)
\Theta_{2,\tau}(v) du dv\,.
\end{equation}
It is clear that $\partial^2_v \Phi_{d,j,m, \tau} $ has at most
finite zeros for fixed $\tau, u$. Thus,
by Lemma \ref{phasetauneg} and van der Corput lemma, we obtain 
\begin{equation}\label{**}
\left|\int e^{i \Phi_{d,j,m, \tau}(u,v)}\Theta_{1, \tau}(u,v)
\Theta_{2,\tau}(v) dv   \right|\leq C\min\left\{1, 2^{-m}|\tau u|^{-1} \right\}
\end{equation}
holds for $|u|\geq 2^{-m}$. Thus we estimate 
$\|\bT_{d,j,m}g\|^2_2 $ by
$$
 C2^{-m}\|g\|_\infty^2+  
C_d\int_{-10}^{10}\int_{-10}^{10} 
\min\left\{1,  2^{-m/2}|\tau u|^{-1/2}\right\}
 \left\| G_\tau\right\|_\infty  du d\tau \,,
$$
which is dominated by
$$
 C_d 2^{- m /2} \|g\|_\infty^2\,.
$$
Henceforth, (\ref{estCorputneg}) follows and we complete the proof.
\end{proof}

We now turn to the proof of Proposition {\ref{prop2neg}}.
For simplicity, we assume $\rho$ is supported on $[1/8, 2]$.
 For any function $q_2
 = a\eta^{-\frac{1}{d-1}} + b\eta \in\mQ_2$, 
we have 
\begin{equation}\label{q1inverseneg}
R_{\Phi_1}\Check{(e^{iq_2})}(x) = \int \wh{\Phi_1}(\eta) e^{ia\eta^{-\frac{1}{d-1}}}e^{i(x+b)\eta} d\eta\,,
\end{equation}
where $|a|\sim 2^m$.  By the stationary phase method, we obtain that the principle part of
(\ref{q1inverseneg}) is 
\begin{equation}\label{prinpartneg}
 \mathcal{P}(q_2)(x)=  C_d |a|^{-1/2} e^{i c_1 a^{(d-1)/d}(x+ b)^{1/d} } 
\wh{\Phi_1}\left(c_2 \left(\frac{x+b}{a}\right)^{-\frac{d-1}{d}}\right)\,,
\end{equation}
where $C_d, c_1, c_2$ are constants depending only on $d$. 
Thus to obtain Proposition {\ref{prop2}}, it suffices to prove that there 
exists a constant $C$ such that
\begin{equation}\label{q1unif23esttineg}
\left |  \ti{\Lambda}_{j,m, n}(f_1, e^{iq_2}, f_3) \right|\leq
 C  2^{-m/4}  \|\check{f_1}\|_2 \| \check{f_3}\|_\infty \,,
\end{equation} 
holds for  all $q_2\in\mQ_2, \check{f_1}\in L^2$, and  $\check{f_3}\in L^\infty$, 
where $ \ti{\Lambda}_{j,m, n}(f_1, e^{iq_2}, f_3)$ equals to 
$$
 2^{\frac{(d-1)j}{2}} \iint \check{f_1}\left(x-2^m t\right)
\mathcal{P}(q_2)\left(2^{(d-1)j}x-2^m t^d\right) 
 \check{f_3}(x) \rho(t) dt dx\,. 
 $$
By a rescaling argument, it suffices to show that 
\begin{equation}\label{rescalneg0}
 \ti\Lambda_{d,j,m}(f, g) \leq C2^{-m/4} \|f\|_2 \| g\|_\infty 
\end{equation}
holds for all $f\in L^2$ and $g\in L^\infty$, 
where $\ti\lambda_{d,j,m}$ is defined by 
\begin{equation}\label{defoftiLadjm}  
\ti\Lambda_{d,j,m}(f, g)=\!\!\!
 \iint \!\!\!f\left(x-2^{(d-1)j}t  \right) g(x)
   e^{iC_{d,j,m}2^m(x-t^d+b/2^m)^{1/d}} \theta_1(x-t^d+b/2^m) \rho(t) dt dx \,.
\end{equation}
Here the constant $C_{d,j,m}$ satisfies $2^{-100}\leq |C_{d,j,m}|\leq 2^{100}$ and 
$\theta_1$ is a bump function supported on $[1/100, 2] $ or $[-2, -1/100]$.
Clearly (\ref{rescalneg0}) is a consequence of Lemma \ref{vdcorputneg}. 
Therefore the proof of Proposition \ref{prop2neg} is completed.  \\

\end{document}